\documentclass[11pt]{article}

\usepackage{latexsym,mathrsfs}
\usepackage{amsmath,amssymb}
\usepackage{amsthm,enumerate,verbatim}
\usepackage{amsfonts}
\usepackage{graphicx}
\usepackage{algorithm}
\usepackage{algorithmic}
\usepackage{url}

\newtheorem{theorem}{Theorem}
\newtheorem{corollary}{Corollary}
\newtheorem{definition}{Definition}
\newtheorem*{definition*}{Definition}

\newtheorem{remark}{Remark}

\newtheorem*{problem*}{Problem}

\numberwithin{equation}{section}
\numberwithin{table}{section}
\numberwithin{figure}{section}

\DeclareMathOperator{\argmin}{argmin}

\newcommand {\mat}  [1] {\left[\begin{array}{#1}}
\newcommand {\rix}      {\end{array}\right]}

\def\real{\mathop{\mathrm{Re}}}

\newcommand{\eproof}{\space
    {\ \vbox{\hrule\hbox{\vrule height1.3ex\hskip0.8ex\vrule}\hrule}}\par}

\def \R{{\mathbb R}}
\def \C{{\mathbb C}}

\title{Finding the nearest positive-real system}
\date{}

\author{Nicolas Gillis$^*$ \qquad Punit Sharma\thanks{The authors acknowledge the support of the ERC (starting grant n$^\text{o}$ 679515).
NG also acknowledges the support of the F.R.S.-FNRS (incentive grant for scientific research n$^\text{o}$ F.4501.16). Email: \{nicolas.gillis, punit.sharma\}@umons.ac.be.} \\ 
Department of Mathematics and Operational Research \\
Facult\'e Polytechnique, Universit\'e de Mons \\
Rue de Houdain 9, 7000 Mons, Belgium\\
 nicolas.gillis@umons.ac.be, punit.sharma@umons.ac.be
}

\begin{document}

\maketitle

\begin{abstract}
The notion of positive realness for linear time-invariant (LTI) dynamical systems, equivalent to passivity, is one of the oldest in system and control theory. In this paper, we consider the problem of finding the nearest positive-real (PR) system to a non PR system: given an LTI  control system defined by $E \dot{x}=Ax+Bu$ and $y=Cx+Du$, minimize the Frobenius norm of $(\Delta_E,\Delta_A,\Delta_B,\Delta_C,\Delta_D)$ such that $(E+\Delta_E,A+\Delta_A,B+\Delta_B,C+\Delta_C,D+\Delta_D)$ is a PR system.
We first show that a system is extended strictly PR if and only if it can be written as a strict port-Hamiltonian system. This allows us to reformulate the nearest PR system problem into an optimization problem with a simple convex feasible set. We then use a fast gradient method to obtain a nearby PR system to a given non PR system, and illustrate the behavior of our algorithm on several examples.
This is, to the best of our knowledge, the first algorithm that computes a nearby PR system to a given non PR system that
(i)~is not based on the spectral properties of related Hamiltonian matrices or pencils,
(ii)~allows to perturb all matrices $(E,A,B,C,D)$ describing the system, and
(iii)~does not make any assumption on the original given system.
\end{abstract}


\section{Introduction}
In this paper, we consider $m$-input $m$-output linear
time-invariant (LTI) control systems of the form
\begin{align}
\label{eq1:sys}
\begin{split}
E\dot{x}(t) = & \; Ax(t)+Bu(t),
\\
y(t) = & \; Cx(t)+Du(t),
\end{split}
\end{align}
on the unbounded interval $t \in [t_0,\infty)$. Here, $A,E \in \R^{n,n}$, $B \in \R^{n,m}$, $C \in \R^{m,n}$, and $D \in \R^{m,m}$
are given matrices, $x(t)$ is the vector of state variables, $u(t)$ is the vector of inputs, and $y(t)$
is the vector of outputs. The linear system is called a \emph{standard system} when $E=I_n$, where $I_n$ is the identity matrix
of size $n\times n$,
and a \emph{descriptor system} when $E$ is not invertible.
We use the matrix quintuple $(E,A,B,C,D)$ to refer to a system in the form~\eqref{eq1:sys}.

As mentioned in \cite{FreJ04}, the restriction to systems~\eqref{eq1:sys} with the same number of inputs and outputs
is necessary to have positive real (PR) systems, which are the focus of this paper. Indeed, positive realness of an LTI dynamical system is equivalent to
passivity, which means that the system does not generate energy: the system~\eqref{eq1:sys} is called
\emph{passive} if there exists a nonnegative scalar valued function $\mathcal V$, called the storage function, such that
$\mathcal V(0)=0$ and the dissipation inequality
\begin{equation}\label{eq:def_pass}
\mathcal V(x(t_1))-\mathcal V(x(t_0)) \leq \int_{t_0}^{t_1} y(t)^T u(t)dt
\end{equation}
holds for all admissible $u$, $t_0$, and $t_1 \geq t_0$; see for example \cite{AndV73,LozBEM13}.
The energy is defined via the inner product of the input and output vectors $u(t)$ and $y(t)$ of the
system hence these vectors need to be of the same length.

Given $(E,A,B,C,D)$, the goal of this paper is to find  $(\Delta_E,\Delta_A,\Delta_B,\Delta_C,\Delta_D)$
with minimum (weighted) Frobenius norm\footnote{
The choice for the (weighted) Frobenius norm is twofold: (i)~it is arguably one of the most popular norm used to measure distances,
and (ii)~it is smooth hence will make the optimization problem easier to tackle. However, our algorithm can easily be extended to any other smooth objective function, e.g., any (weighted) $\ell_p$ norm with $1 < p < +\infty$ (only the computation of the gradient of the corresponding objective function will change).
 } such that
$(E+\Delta_E,A+\Delta_A,B+\Delta_B,C+\Delta_C,D+\Delta_D)$ is a PR system.
We will refer to this problem as the \emph{nearest PR system} problem.
We will also consider the case of nearest standard PR systems  when $E = I_n$ imposing $\Delta_E = 0$.
Since passivity and positive realness are equivalent for LTI systems, the distance to positive realness has
direct  applications in  passive model approximations. For example,
when a real-world problem is approximated by a
model~\eqref{eq1:sys}, the passivity of the physical system may not be preserved, that is, the approximation process
(for example, finite element or finite difference models, linearization, or model order reduction)
makes the passive system
nonpassive. The nonpassive system has to be approximated by a nearby passive system by perturbing
$E,A,B,C$, and $D$.

Several algorithms tackle this problem using the spectral
properties of the related Hamiltonian/skew-Hamiltonian matrices  or pencils
for the input systems that are asymptotically stable, controllable,
observable and almost passive; see~\cite{Tal04,SchT07,WanZKPW10,VoiB11} and the references therein.
The algorithms in~\cite{Tal04} and~\cite{SchT07} impose additional assumptions on the input system, namely $E=I_n$ and $D+D^T$ nonsingular, and are restricted to perturbations of the matrix $C$ only.
These algorithms are based on the displacement of the imaginary eigenvalues of the related
Hamiltonian matrix. The methods in~\cite{WanZKPW10} and~\cite{VoiB11} can deal with general systems (i.e., when $E$ is not identity) by
using the spectral properties of (skew-)Hamiltonian pencils, but they
only allow perturbations in either $B$ or $C$.
In~\cite{BruS13}, authors allow perturbations in all matrices but $E$ and assume that the
system is almost passive. Their approach is based on the perturbation of a general non-dissipative system to enforce dissipativity
using first-order spectral perturbation results for para-Hermitian pencils.
As far as we know, no algorithm exist that does not make any assumption on the input system and that allows perturbations of all matrices $(E,A,B,C,D)$ describing the system.

The nearest PR system problem is complementary with the distance to
nonpassivity for control systems; see~\cite{OveV05} for complex standard systems.
These problems are closely related to the Hamiltonian matrix nearness problems~\cite{AlaBKMM11,GugKL15}.
For example, it is well known~\cite{BoyGFB94,Ant05} that an
asymptotically stable standard system~\eqref{eq1:sys} (i.e., with $E=I_n$)
with positive definite $D+D^T$ is PR if and only if the Hamiltonian matrix
\begin{equation}\label{eq:Ham_matrix}
\mathcal H =\mat{cc} A-B(D+D^T)^{-1}C & -B(D+D^T)^{-1} B^T \\
C^T(D+D^T)^{-1}C & -(A-B(D+D^T)^{-1}C)^T \rix
\end{equation}
has no purely imaginary eigenvalues. Therefore one can use the minimal perturbations
from~\cite{AlaBKMM11} and~\cite{GugKL15} that moves all eigenvalues of $\mathcal H$ away from
the imaginary axis, and find (if they exist) the corresponding perturbations $(\Delta_A,\Delta_B,\Delta_C,\Delta_D)$ making the system passive.
The later step is however in general not possible as it involves dealing with the additional block structure in the Hamiltonian matrix hence solving highly nonlinear matrix equations~\cite{GugKL15}.
It is an open problem to extract system matrices $(A,B,C,D)$ from a given Hamiltonian matrix that has no purely imaginary eigenvalues, that is, to express the Hamiltonian matrix as a matrix of the form~\eqref{eq:Ham_matrix}. \\

In this paper, we compute a nearby PR system to a given non PR system using the set of linear port-Hamiltonian (PH) systems. Our algorithm is based on the generalization of the results from~\cite{GilS16} and~\cite{GilMS17} where authors used the structure of PH systems to find a nearby stable standard system and a nearby stable descriptor system to an unstable one, respectively. As opposed to the previously proposed methods, our algorithm is not based on the spectral properties of Hamiltonian matrices or pencils and can be applied to any given LTI dynamical system.

The paper is organized as follows. In Section~\ref{nota}, we introduce the notation and definitions that will be used throughout the paper.
In Section~\ref{sec:keyresults}, we 
show
that a system is extended strictly PR
if and only if
it can be written as a strict PH system (Theorem~\ref{thm:main_result}).
This allows us to reformulate, in Section~\ref{reformu}, the nearest PR system problem
into an optimization problem with a simple convex feasible set.
In Section~\ref{algosol}, we use a fast gradient method to tackle our reformulation and obtain a nearby PR system to a given non PR system,  for both standard and general systems. We also propose several initialization strategies.
The behavior of the algorithm is analyzed on several examples in Section~\ref{numexp}.

\section{Notation, preliminaries and problem definition} \label{nota}

In the following, 
we denote 
by ${\|\cdot\|}_F$ the Frobenius norm, and by ${}^*$ the complex conjugate transpose.
We write $A\succ 0$ (resp.\@ $A\succeq 0$) if $A$ is symmetric positive definite (resp.\@ semidefinite).
The real part of $s\in \C$
is denoted by $\real{s}$ and $j$ stands for the imaginary number. 

In the next two subsections, we define admissible and PR systems (Sections~\ref{admis} and \ref{posrelsys}).
This allows us to give a formal definition of the nearest PR system problem in Section~\ref{sec:theproblem}. In Section~\ref{phsys}, we
briefly describe PH systems that will be our main tool to reformulate the nearest PR system problem.

\subsection{Admissible systems} \label{admis}

The system~\eqref{eq1:sys} is called \emph{regular} if the matrix pair $(E,A)$ is regular, that is, if
$\operatorname{det}(\lambda E-A)\neq 0$ for some $\lambda \in \mathbb C$, otherwise it is called \emph{singular}.
For a regular matrix pair $(E,A)$, the roots of the polynomial $\operatorname{det}(z E-A)$ are called \emph{finite eigenvalues} of the pencil
$zE-A$ or of the pair $(E,A)$.
A regular pencil $zE-A$ has \emph{$\infty$ as an eigenvalue} if $E$ is singular.

A regular real matrix pair $(E,A)$ (with $E,A\in \R^{n,n}$) can be transformed to \emph{Weierstra\ss\ canonical form} \cite{Gan59a}, that is, there exist nonsingular matrices $W, T \in \C^{n,n}$ such that
\[
E=W\mat{cc}I_q& 0\\0&N\rix T \quad \text{and}\quad A=W \mat{cc}J &0\\0&I_{n-q}\rix T,
\]
where $J \in \C^{q,q}$ is a matrix in \emph{Jordan canonical form} associated with the $q$ finite eigenvalues of
the pencil $z E-A$ and  $N \in \C^{n-q,n-q}$ is a nilpotent matrix in Jordan canonical form  corresponding
to  $n-q$ times the  eigenvalue $\infty$. If $q < n$ and $N$ has degree of nilpotency $\nu \in \{1,2,\ldots\}$, that is, $N^{\nu}=0$ and $N^i \neq 0$ for $i=1,\ldots,\nu-1$, then $\nu$ is called the \emph{index of the pair} $(E,A)$. If $E$
is nonsingular, then by convention the index is $\nu=0$; see for example~\cite{Meh91,Var95}.
The index $\nu$ does not depend on the transformation to canonical form~\cite[Lemma 2.10]{KunM06}.

The  matrix pair $(E,A) \in (\R^{n,n})^2$ is said to be \emph{stable} (resp.\@ \emph{asymptotically stable})
if
all the finite eigenvalues of $zE-A$ are
in the closed (resp.\@ open) left half of the complex plane and those on the
imaginary axis are semisimple.
The  matrix pair $(E,A)$ is said to be \emph{admissible} if it is regular, of index at most one, and
asymptotically stable.
A dynamical system $(E,A,B,C,D)$ in the form~\eqref{eq1:sys} is called (asymptotically) stable
if the matrix pair $(E,A)$ is (asymptotically) stable. Similarly, it is called
admissible if the matrix pair $(E,A)$ is admissible.

\subsection{Positive real systems} \label{posrelsys}

To define positive real systems, throughout this section we
assume that the system~\eqref{eq1:sys} is regular.
The system~\eqref{eq1:sys} can be described by its \emph{transfer function}
$G(s):\C\rightarrow (\C \cup \{\infty\})^{m,m}$, defined by
\begin{equation}\label{tran_fun}
G(s):=C(sE-A)^{-1}B+D,\quad s \in \C.
\end{equation}
Conversely, given a rational function $G(s):\C\rightarrow (\C \cup \{\infty\})^{m,m}$, any
representation of $G(s)$ in the form~\eqref{tran_fun} is called a  realization   of
$G(s)$. A realization is called minimal if the matrices $A$ and $E$ are of smallest
possible dimension.
In this case the poles  of the transfer function $G(s)$ are exactly the eigenvalues of the pencil $zE-A$.

Positive realness is a well-known concept in system, circuit and
control theory. In control theory, the PR systems have a significant role
in stability analysis~\cite{Pop73,AndV73}, see also~\cite{Jos89} and the references therein for applications.
The PR systems have been defined in several different ways in the literature;
see \cite{AndV73,Wen88,LozJ90,SunKS94,IonT87,HadB91,HuaIMS99}
for  standard linear systems, \cite{WanC96,FreJ04,LozBEM13} for continuous-time descriptor systems, and \cite{ZhaLX02} for continuous- and discrete-time descriptor systems.
We follow~\cite{SunKS94} and define the positive realness in the frequency domain as follows.
\begin{definition}\label{def:PR_SPR_ESPR}
The system~\eqref{eq1:sys} is said to be
\begin{enumerate}
\item  \emph{positive real (PR)} if its transfer function $G(s)$ satisfies
\begin{enumerate}
\item $G(s)$ has no pole in $\real{s} >0$, and
\item $G(s)+G(s)^* \succeq 0$ for all $s$ such that $\real{s} > 0$.
\end{enumerate}
\item \emph{strictly positive real (SPR)} if its transfer function $G(s)$ satisfies
\begin{enumerate}
\item $G(s)$ has no pole in $\real{s} \geq0$, and
\item $G(jw)+G(jw)^* \succ 0$ for $w \in [0,\infty)$.
\end{enumerate}
\item \emph{extended strictly positive real (ESPR)} if it is SPR and
$G(j\infty)+G(j\infty)^* \succ 0$.
\end{enumerate}
\end{definition}

Note that the condition $(a)$ in the definition of SPR
is equivalent to the system being asymptotically stable.
An asymptotically stable system~\eqref{eq1:sys} with a minimal realization
is passive (resp.\@ strictly input passive) if and only if it is PR (resp.\@ ESPR).
For more details of these facts, we refer to~\cite{AndV73} and \cite[pp.~174-175]{DesV75}.
Furthermore, ESPR implies SPR, which further implies PR.

Note also that $G(s)=C(sE-A)^{-1}B+D$ is a rational function and has a power series expansion about $s=\infty$
of the form
\begin{equation}\label{eq:series_expan}
G(s)=C(sE-A)^{-1}B+D=\sum_{i=-p}^{\infty} \frac{H_i}{s^i},
\end{equation}
where $H_i$ are real matrices of size $m$. If $s=\infty$ is not a pole of $G(s)$ (i.e., when $E$ is invertible),
then $p=0$ and $G(\infty)=D=H_0$. This implies that
for a standard system $(I_n,A,B,C,D)$ with $D+D^T \succ 0$, the notion of SPR and ESPR are the same, because
$G(j\infty)+G(j\infty)^* \succ 0$ if and only if $D+D^T \succ 0$.
In the descriptor case, in view of~\eqref{eq:series_expan},  we have $p\geq 1$ the order of the pole $s=\infty$.
In this case, $G(\infty)$ (if it exists) is not necessarily equal to $D$. To illustrate this
consider the system
\begin{equation} \label{ex1}
E = \mat{ccc}1 & 0& 0\\0 &0&0 \\ 0&0&0\rix,\quad A=\mat{ccc}-1&0&0\\ 0&1&0 \\ 0&0&1\rix,
B = \mat{c}1\\ 1\\ \alpha \rix,
 C= \mat{c}1 \\ 1 \\ 1 \rix^T, D=\frac{1}{2},
\end{equation}
where $\alpha$ is a constant. This is an admissible system with the transfer function
\[
G(s)=\frac{1}{s+1}-\alpha-\frac{1}{2},
\]
so that
\[
G(jw)+G(jw)^*=\frac{2}{w^2+1} -2\alpha -1.
\]
Therefore, for $\alpha=-\frac{1}{2}$, the system is SPR but not ESPR despite $D \succ 0$.

\subsection{Nearest positive-real system problems} \label{sec:theproblem}

We can now define the nearest system problems that will be studied in the following sections.
Let us formulate the problem in a rather generic way.
\begin{problem*}\label{prob_g}
For a given system $(E,A,B,C,D)$ and a given set $\mathcal{D}$, find the nearest system $(\tilde E,\tilde A,\tilde B,\tilde C, \tilde D) \in \mathcal{D}$ to $(E,A,B,C,D)$, that is, solve
\begin{equation*}\label{def_F}
\inf_{(\tilde E,\tilde A,\tilde B,\tilde C, \tilde D) \in \mathcal{D}}
\mathcal{F}(\tilde A,\tilde B,\tilde C,\tilde D,\tilde E),
\end{equation*}
where
\begin{equation}\label{eq:def_F}
\mathcal{F}(\tilde A,\tilde B,\tilde C,\tilde D,\tilde E) =
{\|A-\tilde A\|}_F^2+{\|B-\tilde B\|}_F^2
+{\|C-\tilde C\|}_F^2+{\|D-\tilde D\|}_F^2
+ {\|E-\tilde E\|}_F^2.
\end{equation}
\end{problem*}
We will consider the following two variants of this problem:
\begin{enumerate}
\item \underline{Nearest PR system} ($\mathcal P$):~
$\mathcal{D} = \mathbb S$
where $\mathbb S$ is the set of all PR systems
$(\tilde E,\tilde A,\tilde B,\tilde C, \tilde D)$.


\item \underline{Nearest ESPR system} ($\mathcal P_e$):~
$\mathcal{D} = \mathbb S_e$ where $\mathbb S_e$ is the set of all admissible ESPR systems
$(\tilde E,\tilde A,\tilde B,\tilde C, \tilde D)$ with $\tilde D+\tilde D^T \succ 0$.

\end{enumerate}

We will also consider the variants of ($\mathcal P$) and ($\mathcal P_e$) for standard systems with the additional constraints that $\tilde{E} = E = I_n$.

These problems are challenging because the feasible sets $\mathbb S$ and $\mathbb S_e$ are unbounded, highly nonconvex, and neither open nor closed. To see this,
let us consider the system from~\eqref{ex1}. This system is ESPR with $\alpha=-1$, and thus $(E,A,B,C,D) \in \mathbb S_e \subseteq \mathbb S$.
For
\[
\Delta_E=\mat{ccc}\epsilon_1 &0&0\\0&\epsilon_2 &0\\0&0&0\rix \; \text{ and } \;
\Delta_A=\mat{ccc}\delta &0&0\\0&0 &0\\0&0&0\rix,
\]
 the  transfer function of the perturbed system $(E+\Delta_E,A+\Delta_A,B,C,D)$ is given by
\[
G(s)=\frac{1}{s(1+\epsilon_1)+(1-\delta)}+\frac{1}{s\epsilon_2-1}-\alpha+\frac{1}{2},
\]
hence
\[
G(jw)+G(jw)^*=\frac{2(1-\delta)}{(1-\delta)^2+w^2(1+\epsilon_1)^2}
-\frac{2}{1+w^2\epsilon_2^2}-2\alpha+1.
\]
For $\delta=\epsilon_1=0$ and $\epsilon_2>0$, the perturbed system is regular,
of index one but has two finite eigenvalues $\lambda_1=-1$ and $\lambda_2=\frac{1}{\epsilon_2} > 0$.
This implies that the system is not stable hence not PR.
This shows that
$\mathbb S$ and $\mathbb S_e$ are not open.
For $\epsilon_1=-\delta$ with $0\leq \delta<1$ and $\epsilon_2=0$, the perturbed system is ESPR.
The perturbed matrix pair $(E+\Delta_E,A+\Delta_A)$ becomes singular as $\delta \rightarrow 1$
so that the perturbed system becomes non PR as $\delta \rightarrow 1$.
This shows that
$\mathbb S$ and $\mathbb S_e$ are not closed.
Further the sets $\mathbb S$  and
$\mathbb S_e$ are nonconvex: considering the systems
\begin{eqnarray*}
\Sigma_1 &=& \Big(I_2,\underbrace{\mat{cc} -0.3 & 10\\0 &-0.3 \rix}_{A_1},\mat{c}1\\1 \rix,\mat{c}1\\0 \rix^T,\frac{1}{2}\Big)
\quad
\text{ and }\\
\Sigma_2 &=& \Big(I_2,\underbrace{\mat{cc} -0.3 & 0\\10 &-0.3 \rix}_{A_2},\mat{c}1\\1 \rix,\mat{c}0\\1 \rix^T,\frac{1}{2}\Big),
\end{eqnarray*}
it is easy to check that $\Sigma_1,\Sigma_2 \in \mathbb S_e \subset \mathbb S$ but
$\gamma \Sigma_1 + (1-\gamma)\Sigma_2 \notin \mathbb S$ for $\gamma=\frac{1}{2}$ since $\frac{1}{2}A_1+ \frac{1}{2}A_2$
has an eigenvalue $\lambda=4.7$ in the right half complex plane.

To address the challenging problems defined above,
we aim to reformulate them so that it is easier to derive optimization algorithms.
An important property of such a reformulation is that the projection onto the feasible set can be performed efficiently.
Such a reformulation exists and can be obtained by extending the results from~\cite{GilS16}
(resp.\@ \cite{GilMS17}) that used PH systems for computing the nearest stable matrix (resp.\@ matrix pair).

\subsection{Port-Hamiltonian systems} \label{phsys}

A linear time-invariant input-state-output \emph{port-Hamiltonian (PH)} system can be written as
\begin{align}
\label{eq:phsystem}
\begin{split}
E\dot{x}(t) = & \; (J-R)Qx(t) + (F-P)u(t), \\
y(t) = & \; (F+P)^T Qx(t) + (S+N)u(t),
\end{split}
\end{align}
%
where the following conditions must be satisfied:
\begin{itemize}

\item The matrix $Q \in \R^{n,n}$ is invertible, $E \in \R^{n,n}$, and  $Q^TE=E^TQ \succeq 0$. The function $x\rightarrow \frac{1}{2}x^T Q^TE x$ is the \emph{Hamiltonian} and
describes the energy of the system.

\item The matrix $J^T=-J \in \R^{n,n}$
is the structure matrix that describes flux among energy storage elements.

\item The matrix $R \in \R^{n,n}$
with $R\succeq 0$ is the dissipation matrix and describes the energy dissipation/loss in the system.

\item The matrices $F\pm P \in \R^{n,m}$ are the port matrices describing the manner in which energy enters and exits
the system.

\item The matrix $S+N$, with $0 \preceq S \in \R^{m,m}$ and $N^T=-N \in \R^{m,m}$,  describes
the direct feed-through from input to output.

\item The matrices $R$, $P$ and $S$ satisfy
\[
K=\mat{cc}R &P\\P^T &S \rix \succeq 0.
\]
\end{itemize}

In the following, we will refer to $K$ as the \emph{cost matrix} of the PH system, because
it corresponds to the cost matrix of an
infinite horizon linear quadratic optimal control problem. For $K \succ 0$, we refer to~\eqref{eq:phsystem} as a \emph{strict PH system}.
We note that this definition of
PH system is slightly more restrictive than that of PH systems in \cite{BeaMXZ17_ppt},
where it is not required for the matrix $Q$ to be invertible.

PH systems generalize the classical Hamiltonian systems and
recently have received a lot attention in energy based modeling;
see \cite{Sch06,Sch13,SchM95,SchM02,SchM13,Sch2014port} for some major references.
The Hamiltonian $\mathcal H(x)=\frac{1}{2}x^TQ^TEx$
defines a storage function associated with the supply rate $y(t)^T u(t)$, and satisfies
\begin{equation} \label{HHyu}
\mathcal H(x(t_1))-\mathcal H(x(t_0)) \leq \int_{t_0}^{t_1} y(t)^T u(t)dt,
\end{equation}
which guarantees the passivity of the system; see~\eqref{eq:def_pass}.

 We note that regular PH systems are always stable,~see~\cite[Lemma 3.1]{MehMS16} for
standard PH systems and~\cite[Lemma 2]{GilMS17} for descriptor PH systems:
the matrix pair $(E,(J-R)Q)$ is a so-called dissipative Hamiltonian matrix pair.
In particular, if $R \succ 0$, then $(E,(J-R)Q)$ is admissible; see \cite{GilMS17} for more details, and
\cite{MehMS16} and \cite{MehMS17} for various structured distances to asymptotic stability for
complex PH systems and real PH systems, respectively.

\section{Key results for positive real systems} \label{sec:keyresults}

In this section, we study the link between PR systems and PH systems. The main result of this section,
which is the main theoretical result of this paper, is to prove that
a system is ESPR with $D+D^T \succ 0$ if and only if
it can be written as a strict PH system; see Theorem~\ref{thm:main_result} at the end of the section.

The positive realness of a system~\eqref{eq1:sys} can be characterized in terms
of solutions $X$ to the following linear matrix inequalities (LMIs):
\begin{equation}\label{eq:LMI1}
\mat{cc} A^TX +X^T A & X^TB-C^T \\B^TX-C & -D-D^T \rix \preceq 0
\quad \text{and} \quad E^TX=X^TE \succeq 0.
\end{equation}
%
We have the following result.
\begin{theorem}[\cite{FreJ04}, Theorem~3.1] \label{thm:suff_PR}
Consider a regular system $(E,A,B,C,D)$ in the form~\eqref{eq1:sys}. If the LMIs~\eqref{eq:LMI1} have a solution $X \in \R^{n,n}$, then
$(E,A,B,C,D)$ is PR.
\end{theorem}

The converse of Theorem~\ref{thm:suff_PR} is true with some additional assumptions.
In fact the positive real lemma for standard systems~\cite{AndV73} proves that
if a system is PR and minimal, then the existence of a solution to the LMIs~\eqref{eq:LMI1} is also necessary.
Similarly, with an additional condition, the positive real lemma for descriptor systems~\cite{FreJ04} proves that  the existence of a solution to the LMIs~\eqref{eq:LMI1} is also necessary for positive realness.

Theorem~\ref{thm:suff_PR} gives an alternative way (compared to the one described in~\eqref{HHyu})
to show that every PH system is positive real by providing an explicit solution to~\eqref{eq:LMI1}, as shown in the following theorem
which is a generalization of~\cite[Theorem 7.1]{Sch2014port} where only standard systems are considered.

\begin{theorem}\label{cor1}
Every regular PH system in the form~\eqref{eq:phsystem} is PR.
\end{theorem}
\proof Let $(E,A,B,C,D)$ be a regular PH system with
$A=(J-R)Q$, $B=F-P$, $C=(F+P)^TQ$ and $D=S+N$, where $J^T=-J$,
$N^T=-N$, $\mat{cc}R & P\\P^T & S \rix \succeq 0$, $Q$ is invertible and $E^TQ=Q^TE \succeq 0$.
By Theorem~\ref{thm:suff_PR}, to prove that this system is PR, it suffices to prove the existence of a solution $X$ to the LMIs~\eqref{eq:LMI1}.
It turns out that $X=Q$ is one. In fact, we have
\begin{align*}
&\mat{cc} A^TQ +Q^T A & Q^TB-C^T \\B^TQ-C & -D-D^T \rix\\
&=\mat{cc} ((J-R)Q)^TQ +Q^T (J-R)Q & Q^T(F-P)-((F+P)^TQ)^T \\(F-P)^TQ-(F+P)^TQ & -(S+N)-(S+N)^T \rix\\
&= -2 \mat{cc} Q^TRQ & Q^TP \\P^TQ & S \rix = -2\mat{cc}Q^T &0\\0 &I_m\rix
\mat{cc} R & P \\P^T & S \rix
\mat{cc}Q & 0\\0 &I_m\rix \preceq 0,
\end{align*}
since $\mat{cc}R & P\\P^T & S \rix \succeq 0$.  
\eproof
Note that the proof of Theorem~\ref{cor1} does not require $Q$ to be invertible.
In the following, a necessary and sufficient condition for a system in the form~\eqref{eq1:sys}
to be ESPR is obtained in terms of the existence of a solution of the LMIs~\eqref{eq:LMI1}. We will use this
result to characterize the set of all admissible ESPR systems in terms of PH systems.

\begin{theorem}[\cite{ZhaLX02}, Theorem~2] \label{thm:espr_equi}
Let $(E,A,B,C,D)$  define a system~\eqref{eq1:sys}. Then
it is admissible, ESPR and satisfying $D+D^T \succ 0$
if and only if
there exists a solution $X$ to the LMIs
\begin{equation}\label{eq:LMI2}
\mat{cc} A^TX +X^T A & X^TB-C^T \\B^TX-C & -D-D^T \rix \prec 0
\quad \text{and} \quad E^TX=X^TE \succeq 0.
\end{equation}
\end{theorem}

Using Theorem~\ref{cor1} and Theorem~\ref{thm:espr_equi}, we prove the following result.
\begin{theorem}\label{thm:ph_necc}
Every PH system in the form~\eqref{eq:phsystem} with a positive definite cost matrix
is admissible and ESPR.
\end{theorem}
\proof
Let $\Sigma=(E,(J-R)Q,(F-P),(F+P)^TQ,S+N)$ be a PH system in the form~\eqref{eq:phsystem}
with the cost matrix $K=\mat{cc}R &P\\P^T &S \rix \succ 0$.
Following the same steps as in Theorem~\ref{cor1}, one can show that $X=Q$ satisfies the LMIs in~\eqref{eq:LMI2} because $K \succ 0$ and $Q$ is invertible.
Hence, by Theorem~\ref{thm:espr_equi}, $\Sigma$ is admissible and ESPR.
\eproof

 We note that the admissibility of the PH system $\Sigma$ in Theorem~\ref{thm:ph_necc} also follows
by the fact that $(E,(J-R)Q)$ is a DH matrix~\cite[Theorem 4]{GilMS17}. 
In order to show that the converse of Theorem~\ref{thm:ph_necc} is also true, we define the \emph{PH-form}
for a system~\eqref{eq1:sys}.

\begin{definition}{\rm
A system $(E,A,B,C,D)$ is said to admit a \emph{port-Hamiltonian form (PH-form)} if
there exists a PH system as defined in~\eqref{eq:phsystem} such that
\[
A=(J-R)Q,\quad B=F-P, \quad C=(F+P)^TQ, \quad \text{and}\quad D=S+N .
\]
}
\end{definition}

In view of Theorem~\ref{thm:ph_necc}, if $(E,A,B,C,D)$ admits a  PH-form
with positive definite cost matrix, then it is admissible and ESPR.
Similarly, by Theorem~\ref{cor1}, it follows that
every  regular system $(E,A,B,C,D)$ that admits a PH-form is PR.
However, the converse, that is, every PR system admits a PH-form is not true,
as there exist PR systems with $D+D^T \prec 0$; for instance~\eqref{ex1} with $\alpha = -\frac{3}{2}$ and replacing $D=-1/2$.

We now show that whenever the LMIs~\eqref{eq:LMI1} have an invertible solution,
the system $(E,A,B,C,D)$ admits a PH-form. This will imply that minimal PR standard systems and ESPR systems admit a PH-form (Corollaries~\ref{corsPH1} and \ref{thm:PR_strict_suff}).

\begin{theorem}\label{lem:PH_suff}
Let $\Sigma = (E,A,B,C,D)$ be a system in the form~\eqref{eq1:sys}.
If the LMIs~\eqref{eq:LMI1} have an invertible solution $X \in \R^{n,n}$, then
$\Sigma$ admits a PH-form.
\end{theorem}
\proof Let $X$ be an invertible solution of the LMIs~\eqref{eq:LMI1}.
Define
\begin{equation*}
 J:=\frac{AX^{-1}-(AX^{-1})^T}{2}, \quad R:=-\frac{AX^{-1}+(AX^{-1})^T}{2}, \quad Q:=X, \quad
S:=\frac{1}{2}(D+D^T), \vspace{-0.4cm}
\end{equation*}
\begin{equation}
 N:=\frac{1}{2}(D-D^T),\quad F:=\frac{1}{2}(B+X^{-1}C^T), \quad \text{and}\quad
P:=\frac{1}{2}(-B+X^{-1}C^T).  \label{DHformcstr}
\end{equation}
 Let us show that the matrices $J,R,Q,F,P,N$ and $S$ provide a PH-form for $\Sigma$. We have
\[
(J-R)Q=A,\quad F-P=B, \quad (F+P)^TQ =C,\quad \text{and}\quad S+N=D.
\]
Further, we have that $E^TQ \succeq 0$ (using the second LMI in~\eqref{eq1:sys}), $J^T=-J$, $N^T=-N$, and
\begin{eqnarray*}
K&=&\mat{cc}R &P\\P^T & S \rix = -\frac{1}{2}\mat{cc} AX^{-1}+X^{-1}A^T & -B+X^{-1}C^T \\ -B^T+CX^{-1} & -D-D^T \rix \\
&=& -\frac{1}{2}\mat{cc} -X^{-1} & 0 \\0& I_m \rix
\mat{cc} A^TX +X A & XB-C^T \\B^TX-C & -D-D^T \rix
\mat{cc} -X^{-1} & 0 \\0& I_m \rix
\succeq 0,
\end{eqnarray*}
which follows from the first LMI in~\eqref{eq:LMI1}.
\eproof

For standard systems,
\cite[Corollary~2]{BeaMX15_ppt} shows that
every minimal PR system $(I_n,A,B,C,D)$ is equivalent to a system in PH-form, that is,
there exist invertible matrices  $T \in \R^{n,n}$ and $V \in \R^{m,m}$ such that the transformed system
\[
(I_n,T^{-1}AT,T^{-1}BV,V^TCT,V^TDV)
\]
 admits a PH-form.
Theorem~\ref{lem:PH_suff} implies a stronger result: a minimal PR standard system itself admits a  PH-form (no transformation is necessary).

\begin{corollary} \label{corsPH1}
If the system  $(I_n,A,B,C,D)$ is
minimal and PR, then it admits a  PH-form.
\end{corollary}
\proof
This follows from the positive real lemma for minimal PR standard systems~\cite[p.363]{Wil72} (which guarantees the existence of an invertible solution $X$ of the LMIs~\eqref{eq:LMI1}) and Theorem~\ref{lem:PH_suff}.
\eproof

We note that as opposed to a standard PR system~\cite[p.363]{Wil72}, minimality of a PR
system does not guarantee the solvability of the LMIs~\eqref{eq:LMI1} in the descriptor case.
For this to hold, an additional condition
that $D+D^T \succeq \lim_{s \rightarrow \infty} \left(G(s)+G(s)^T\right)$
is also needed~\cite[Theorem 3.2]{FreJ04}.

\begin{corollary} \label{thm:PR_strict_suff}
Every admissible and ESPR dynamical system $(E,A,B,C,D)$ with $D+D^T \succ 0$ admits a  PH-form with positive definite cost matrix.
\end{corollary}
\proof By Theorem~\ref{thm:espr_equi}, there exists a solution $X$ to the LMIs
\begin{equation}\label{proof:suff_ESPR}
\mat{cc} A^TX +X^T A & X^TB-C^T \\B^TX-C & -D-D^T \rix \prec 0
\quad \text{and} \quad E^TX=X^TE \succeq 0.
\end{equation}
This implies that $A^TX +X^T A \prec 0$, and therefore $X$ is invertible.
The remainder of the proof follows using the same arguments that of Theorem~\ref{lem:PH_suff} with
the solution $X$ of the LMIs~\eqref{proof:suff_ESPR}.
\eproof

In the following, we summarize several equivalent characterizations of a system
to be admissible and ESPR. 

\begin{theorem}\label{thm:main_result}
Let $\Sigma=(E,A,B,C,D)$ be a system in the form~\eqref{eq1:sys}. Then the following are equivalent.
\begin{enumerate}
\item  $\Sigma$ is admissible and ESPR with $D+D^T \succ 0$.
\item There exists a solution $X$ to the LMIs
\begin{equation*}\label{eq:LMI3}
\mat{cc} A^TX +X^T A & X^TB-C^T \\B^TX-C & -D-D^T \rix \prec 0
\quad \text{and} \quad E^TX=X^TE \succeq 0.
\end{equation*}
\item $\Sigma$ admits a  PH-form with positive definite cost matrix.
\end{enumerate}
\end{theorem}
\proof
This follows from Theorems~\ref{thm:espr_equi} and~\ref{thm:ph_necc}, and Corollary~\ref{thm:PR_strict_suff}.
\eproof

In the next section, we reformulate the nearest ESPR system problem $(\mathcal P_e)$ using the
PH-form for an admissible ESPR system
$(E,A,B,C,D)$ with $D+D^T \succ 0$.
As mentioned in Section~\ref{posrelsys}, for a standard ESPR system $(I_n,A,B,C,D)$ we have that $D+D^T \succ 0$,
thus the condition $D+D^T \succ 0$ for standard systems is redundant.
However, the PH-form characterization of an admissible ESPR descriptor system depends
on the existence of a solution of the LMIs~\eqref{eq:LMI2} when $D+D^T \succ 0$.
This justifies the restriction $D+D^T \succ 0$ on defining the set $\mathbb S_e$ for the nearest ESPR system problem
in Section~\ref{sec:theproblem}.

\section{Reformulation of the nearest PR system problems} \label{reformu}

In this section, we exploit the results obtained in the previous section and
present a new framework based on PH systems to attack $(\mathcal P)$ and $(\mathcal P_e)$
defined in Section~\ref{sec:theproblem}, as well as their variants for standard systems. \\

Let us define the following two sets:
\begin{itemize}
\item The set $\mathbb S_{PH}$
containing all systems $(E,A,B,C,D)$ in PH-form, that is,
\begin{eqnarray*}
\mathbb S_{PH} &:=& \left\{(E,A,B,C,D) \; | \ (E,A,B,C,D) \text{ admits a PH-form}\right\} \\
&=& \Big\{ (E,(J-R)Q,F-P,(F+P)^TQ,S+N) \; \Big|  \;
J^T=-J, N^T=-N, \\
&& \hspace{3cm} E^TQ \succeq 0, Q \text{ invertible},
K=\mat{cc}R &P \\
P^T & S \rix \succeq 0 \Big\}.
\end{eqnarray*}

\item The set $S_{PH}^{\succ 0} \subset \mathbb S_{PH}$
containing all systems $(E,A,B,C,D)$ in strict PH-form, that is,
\begin{eqnarray*}
   \mathbb S_{PH}^{\succ 0}  := \Big\{(E,(J-R)Q,F-P,(F+P)^TQ,S+N) \in \mathbb S_{PH}
	\; \Big| \; K 
	\succ 0 \Big\}.
\end{eqnarray*}
By Theorem~\ref{thm:main_result}, $\mathbb S_e =  \mathbb S_{PH}^{\succ 0}$.
\end{itemize}

We have discussed in Section~\ref{sec:keyresults} that the set $\mathbb S_e$ of all
ESPR systems is neither open nor closed and clearly the PH characterization of
$\mathbb S_e$ does not change this.
In fact, the sets $\mathbb S_{PH}$ and $\mathbb S_{PH}^{\succ 0}$ are neither closed (due to the constraint that $Q$ is invertible)
nor open (due to the constraint $E^TQ \succeq 0$).

Since we want to work with a set onto which it is easy (and possible) to project,
 we consider the closure $\overline{\mathbb S_{PH}}$ of $\mathbb S_{PH}$ which is equal to the set $\mathbb S_{PH}$ except that $Q$ can be singular. Moreover, we have that $\overline{\mathbb S_{PH}} = \overline{\mathbb S_{PH}^{\succ 0}}$.
Therefore the values of the infimum over the sets $\mathbb S_{PH}$, $\mathbb S_{PH}^{\succ 0}$, and
$\overline{\mathbb S_{PH}}$ are the same. We have the following result.
\begin{theorem}\label{thm:reform_in_ph}
Let $(E,A,B,C,D)$ be a system in the form~\eqref{eq1:sys} and $\mathcal{F}$ be defined as in~\eqref{eq:def_F}.
Then
\begin{equation} \label{eq:reform_in_ph}
\inf_{ (M,(J-R)Q,F-P,(F+P)^TQ,S+N) \in \overline{\mathbb S_{PH}} }  \quad
\mathcal{F}((J-R)Q,F-P,(F+P)^TQ,S+N,M)
\end{equation}
coincides with the infimum of $(\mathcal P_e)$  while it is
is an upper bound for the infimum of $(\mathcal P)$.
\end{theorem}
\begin{proof}
This follows directly from the fact that $\mathbb S_e =  \mathbb S_{PH}^{\succ 0}$ and $\mathbb S_e \subseteq \mathbb S$.
\end{proof}
We will refer to~\eqref{eq:reform_in_ph} as the nearest PH system problem.
The same result holds for the variants of $(\mathcal P)$ and $(\mathcal P_e)$ for standard systems
since the only difference is that $M$ is imposed to be equal to $E=I_n$.

\begin{remark}\label{re:revision1}{\rm
We note that for standard systems we have $\overline{\mathbb S_{PH}} \subseteq \mathbb S$,
this is due to Theorem~\ref{cor1} as its proof does not require $Q$ to be invertible.
In the descriptor case, $\overline{\mathbb S_{PH}}$ could contain systems which are
not regular. This shows that in this case a feasible solution
of~\eqref{eq:reform_in_ph} may not be a PR system. To rule out such situations,
one can impose the matrix $R$ to satisfy $R\succeq \delta I_n$ for some fixed
small $\delta >0$. This does not make the problem more complicated as
the projection is still straightforward, but gives a nearby regular
descriptor PH system~\cite[Lemma 1]{GilMS17} (hence a PR system, see Theorem~\ref{cor1})
to a given system.
}
\end{remark}

\begin{remark}\label{re:revision2}{\rm
Although the value of the infimum in~\eqref{eq:reform_in_ph} coincides with the infimum of $(\mathcal P_e)$,
the solution of~\eqref{eq:reform_in_ph} may not solve problem $(\mathcal P_e)$,
as the solution found may not even be PR; see Remark~\ref{re:revision1}.
However, it is possible to obtain a nearby strict PH system (hence
admissible and ESPR system with $D+D^T \succ 0$, see Theorem~\ref{thm:main_result})
by rewriting~\eqref{eq:reform_in_ph} 
using lower bounds on the eigenvalues of matrix $Z$ and $K$; see also Remark~\ref{rem:strictPR}.
}
\end{remark}

\section{Algorithmic solution to the nearest PH system problem} \label{algosol}

In this section, we propose an algorithm to tackle~\eqref{eq:reform_in_ph}.
We analyze separately standard systems when $E=I_n$ and $E$ is not subject to perturbation,
and general systems when $E$ is subject to perturbation.

\subsection{Standard systems} \label{stansys}

For standard systems, $M=E=I_n$ and~\eqref{eq:reform_in_ph} can be simplified as follows
{\small{\begin{align}
 \inf_{J,R,Q,F,P,S,N} & 
 {\|A-(J-R)Q\|}_F^2 + {\|B-(F-P)\|}_F^2+{\|C-(F+P)^TQ\|}_F^2+{\|D-(S+N)\|}_F^2
\nonumber \\
& \text{ such that } \quad J^T=-J, Q\succeq 0, N^T=-N \text{ and } \mat{cc}R &P\\P^T&S\rix\succeq 0 . \label{def:dist_sph_1}
\end{align}}}
For any given square matrix $Z$ and any skew-symmetric matrix $X$, we have
\begin{equation}\label{eq:nearestskew-sym}
{\left\|Z-X\right\|}_F^2= {\left\|\frac{(Z+Z^T)}{2}+\frac{(Z-Z^T)}{2}-X\right\|}_F^2= {\left\|\frac{(Z+Z^T)}{2}\right\|}_F^2 +
{\left\|\frac{(Z-Z^T)}{2}-X\right\|}_F^2,
\end{equation}
since symmetric and skew-symmetric matrices are orthogonal (their inner product is zero).
Therefore the infimum in~\eqref{eq:nearestskew-sym} over all skew-symmetric matrices $X$ is attained when $X=\frac{Z-Z^T}{2}$, that is,
\begin{equation} \label{projjmjt}
\min_{X^T=-X} {\|Z-X\|}_F
\quad = \quad
 {\left\|Z -  \frac{(Z-Z^T)}{2}   \right\|}_F .
\end{equation}
This implies that the optimal $N$ in \eqref{def:dist_sph_1}
is given by $\frac{D-D^T}{2}$ since $S$ is symmetric, so that \eqref{def:dist_sph_1} can be simplified to
\begin{align}
\inf_{J,R,Q,F,P,S}
& \mathcal{G}(J,R,Q,F,P,S) \quad \text{ such that }
\quad J^T=-J, Q\succeq 0 \text{ and } \mat{cc}R &P\\P^T&S\rix\succeq 0 ,  \label{def:dist_sph_2}
\end{align}
where
\begin{align*}
\mathcal{G}(J,R,Q,F,P,S)
& = {\|A-(J-R)Q\|}_F^2 + {\|B-(F-P)\|}_F^2 \\
& \quad \quad + {\|C-(F+P)^TQ\|}_F^2+{\left\|\frac{D+D^T}{2}-S\right\|}_F^2 .
\end{align*}
Similarly as it was done in~\cite{GilS16} to find the nearest stable matrix to an unstable one,
we develop a fast projected gradient method (FGM) to solve~\eqref{def:dist_sph_2}.
FGM has the advantage to be in general much faster than the standard projected gradient method~\cite{nes83} \cite[p.90]{Nes04}
(see Section~\ref{numexp} for some examples), even in the non-convex case~\cite{ONW17}, while being relatively simple to implement as long as one can
\begin{itemize}
\item Compute the gradient: all the terms in the objective function are of the form $f(X) = {\|AX-B\|}_F^2$
whose gradient is $\nabla_X f(X) = 2A^T(AX-B)$.

\item Project onto the feasible set: the projection onto the set $\{X \,:\, X=-X^T\}$ is given in~\eqref{projjmjt}
while projection onto the set of positive semidefinite matrices can be computed using an eigenvalue decomposition~\cite{Hig88b}.
\end{itemize}

For the sake of completeness, we describe in Algorithm~\ref{fastgrad} the variant of FGM we use.
Note that we have included a restarting procedure which is necessary in our case since the objective function is not convex hence FGM without restart is not guaranteed to converge~\cite{GL16}. We have observed that, in most cases, Algorithm~\ref{fastgrad} does not need to restart very often
 (on the numerical examples presented in Section~\ref{numexp}, it restarts in average less than every 1000 iterations). 
 \algsetup{indent=2em}
\begin{algorithm}[ht!]
\caption{Fast Gradient Method (FGM) with Restart} \label{fastgrad}
\begin{algorithmic}[1]
\REQUIRE
The (non-convex) function $f(x)$,
the feasible set $\mathcal{X}$,
an initial guess $x \in \mathcal{X}$,
a parameter $\alpha_1 \in (0,1)$,
lower bound for the step length $\underline{\gamma}$.

\ENSURE An approximate solution $x$ to the problem $\argmin_{z \in \mathcal{X}} f(z)$.  \medskip

\STATE $y = x$ ; initial step length $\gamma > \underline{\gamma}$.

\FOR{$k = 1, 2, \dots$}

\STATE  \emph{\% Keep the previous iterate in memory.} 

\STATE $\hat x = x$. \hspace{1cm}

\STATE \emph{\% Project the gradient step onto $\mathcal{X}$, where $\mathcal{P}_{\mathcal{X}}(z) = \argmin_{\tilde{z} \in \mathcal{X}} \|z-\tilde{z}\|$.
} 

\STATE $x = \mathcal{P}_{\mathcal{X}}  \big(  y - \gamma \nabla f(y) \big)$.

\STATE    \emph{\% Check if the objective function $f$ has decreased, otherwise decrease the step length.} 

\WHILE { $f(x)  > f(\hat x)$ and $ \gamma \geq \underline{\gamma}$ }

\STATE $\gamma = \frac{2}{3} \gamma$.

\STATE $x = \mathcal{P}_{\mathcal{X}}  \big(  y - \gamma \nabla f(y) \big)$.

\ENDWHILE

\STATE  \emph{\% If the step length has reached the lower bound ($f$ could not be decreased from $y$), reinitialize $y$ (the next step will be a regular gradient descent step).} 

\IF { $\gamma < \underline{\gamma}$ }

\STATE Restart fast gradient: $y = x$; $\alpha_{k} = \alpha_{1}$.

\STATE   Reinitialize $\gamma$ at the last value for which it allowed decrease of $f$.

\ELSE

\STATE $\alpha_{k+1} = \frac{1}{2} \left(  \sqrt{ \alpha_{k}^4 + 4 \alpha_{k}^2 } - \alpha_{k}^2 \right)$,
$\beta_k =  \frac{\alpha_{k} (1-\alpha_{k})}{\alpha_{k}^2 + \alpha_{k+1}}$.

\STATE $y = x + \beta_k \left(x - \hat x\right)$.

\ENDIF

\STATE $\gamma = 2 \gamma$.

\ENDFOR

\end{algorithmic}
\end{algorithm}

In our implementation, we have also added the possibility to give a different importance to each term in the objective function using nonnegative weights $w_i \geq 0$ ($1 \leq i \leq 4)$ and minimize
{\small{\begin{equation}\label{eq:weighted_obj}
w_1 {\|A-(J-R)Q\|}_F^2 + w_2 {\|B-(F-P)\|}_F^2+ w_3 {\|C-(F+P)^TQ\|}_F^2+ w_4 {\left\|\frac{D+D^T}{2}-S\right\|}_F^2 .
\end{equation}}}

\paragraph{Parameter settings}

For the initial step length, we use $\gamma = 1/L$ where $L={\|Q\|}_2^2$ is the Lipschitz constant
of the gradients of $\mathcal{G}$ with respect to $J$ (and $R$). The reason for this choice is that this step length would guarantee the decrease of the objective function if we would only update $J$ (or $R$) since the subproblem in $J$ (and $R$) is convex. Note that the Lipschitz constant
of the gradients of $\mathcal{G}$ with respect to $F$ (resp.\@ $S$) is $L+1$ (resp.\@ 1). Hence, except maybe for the variable $Q$, this value of $\gamma$ has a good order of magnitude while being simple to compute. In fact, computing the Lipschitz constant of the full gradient of $\mathcal{G}$ is nontrivial and computationally more expensive while the Hessian of $\mathcal{G}$ is mostly block diagonal (only the variable $Q$ appears with other variables).
We choose $\alpha_1=0.5$ which seems to work well in most cases, although FGM can be quite sensitive to this parameter and it is sometimes rewarding to try different values. In fact, even in the convex case, there is, as far as we know, no theoretical recommendation on how to choose this value; it is problem dependent.
(We also tried $\alpha_1=0.1, 0.9$ which performed in average slightly worse than $\alpha_1=0.5$.)

\begin{remark}[Closed form for $F$]{\rm
The optimal solution for the variable $F$ in~\eqref{def:dist_sph_2} can be written in closed form:
\begin{align}
\hat F & =
\argmin_{F} {\|B - (F-P)\|}_F^2 + {\|C^T - Q^T(F+P)\|}_F^2 \nonumber \\
 & =  ( I_n + QQ^T )^{-1} \left( P+B + QC^T + QQ^T P \right), \label{formulF}
\end{align}
since
{\small{
\begin{equation*}
\frac{1}{2} \nabla_F ({\|F - (P+B)\|}_F^2 + {\|Q^T F - (C^T + Q^TP)\|}_F^2 )
= (F-(P+B)) + Q(Q^T F - (C^T + Q^TP)).
\end{equation*}}}
However, we did not inject $\hat F$ in~\eqref{def:dist_sph_2} as it makes the objective function very complicated; in particular because of the term
${\|C^T - Q^T(\hat F +P)\|}_F^2$.
}
\end{remark}

\subsection{General systems}  \label{sec:gen_algo}

Similarly as for standard systems in~\eqref{def:dist_sph_2}, \eqref{eq:reform_in_ph} can be simplified to
\begin{align}
\inf_{J,R,Q,M,F,P,S} & \mathcal{G}(J,R,Q,F,P,S) + {\|E - M\|}_F^2 \label{def:dist_dph_1} \\
& \quad \text{ such that } J^T=-J,
M^T Q \succeq 0
\text{ and }
\mat{cc}R &P\\P^T&S\rix\succeq 0 . \nonumber
\end{align}
As opposed to~\eqref{def:dist_sph_2}, it is difficult to project on the feasible domain of~\eqref{def:dist_dph_1}
because of the coupling constraint $M^T Q \succeq 0$. Moreover, this constraint was observed to get standard optimization schemes stuck in suboptimal solutions; see~\cite[Example 3]{GilMS17} for an example.
Following the strategy used in~\cite{GilMS17}, we introduce the variable $Z = M^T Q$ so that $M^T = Z Q^{-1}$.
This allows us to reformulate~\eqref{def:dist_dph_1} into an equivalent optimization problem with a
simpler feasible set:
\begin{align}\label{eq:equivalent1}
\inf_{J,R,Q,Z,F,P,S} 
& \mathcal{G}(J,R,Q,F,P,S) + {\|E^T - Z Q^{-1} \|}_F^2 \\
& \text{ such that } J^T=-J, Z \succeq 0 \text{ and } \mat{cc}R &P\\P^T&S\rix\succeq 0 ,\nonumber
\end{align}
for which we have implemented Algorithm~\ref{fastgrad} (the gradient of ${\|E^T - Z Q^{-1} \|}_F^2$ with respect to $Q$ is given in~\cite[Appendix~A]{GilMS17}).
As for standard systems, we have also added the possibility to use weights for the different terms in the objective function as
in~\eqref{eq:weighted_obj} adding the term $w_5 {\big\|E^T-ZQ^{-1}\big\|}_F^2$ with $w_5 \geq 0$.
For the initial step length, we use $\gamma = 1/L$ where $L=\max({\|Q\|}_2^2,{\|Q^{-1}\|}_2^2)$ is the maximum between the Lipschitz constants of the gradients of $\mathcal{G}$ with respect to $J$, $R$ and $Z$.

\subsection{Initializations} \label{initsec}

Since we are dealing with non-convex optimization problems, it is expected that choosing good initial points will be crucial to obtain good solutions. This will be confirmed in the numerical experiments. In the next two subsections, we propose several initialization strategies.
We believe that designing other initialization strategies is an important direction for further research; in particular taking advantage of the particular structure of the problem at hand.

\subsubsection{Standard initialization} \label{staninit}

The first initialization, which we refer to as `standard', uses $Q = I_n$ and $P = 0$. For these values of $Q$ and $P$, the optimal solutions for the other variables can be computed explicitly:
\[
J = \big(A-A^T\big)/2,
R = \mathcal{P}_{\succeq} \big((-A-A^T)/2\big),
S = \mathcal{P}_{\succeq} \big((D^T+D)/2\big),
F = \big(B+C^T\big)/2,
\]
and $Z = \mathcal{P}_{\succeq }(E^T)$ for general systems. The notation $\mathcal P_{\succeq }(X)$ stands for the projection of a matrix $X$
on the cone of positive semidefinite matrices.
This initialization has the advantage to be very simple to compute while working reasonably well in many cases;
see Section~\ref{numexp} for numerical experiments.

\subsubsection{LMI-based initializations} \label{LMIinit}

Given a system that does not admit a PH-form, the LMIs~\eqref{eq:LMI1} will not have a solution. However, since we are looking for a nearby system that will admit a solution to these LMIs, it makes sense to try to find a solution $X$ to nearby LMIs.
We propose the following to relax the LMIs~\eqref{eq:LMI1}:
\begin{align}
\min_{\delta, X} & \quad \delta^2 \nonumber \\
\text{ such that } & \quad
\mat{cc} -A^TX -X^T A & C^T-X^TB \\ C-B^TX & D+D^T \rix + \delta I_{n+m} \succeq 0,  \label{initLMIs} \\
& \quad  E^T X  + \delta I_n \succeq 0. \nonumber
\end{align}
Let us denote $(\delta^*,X^*)$ an optimal solution of~\eqref{initLMIs}.
By Theorem~\ref{lem:PH_suff}, if $\delta^* = 0$ and $X^*$ is invertible,
then the system $(E,A,B,C,D)$ admits a PH-form that can be constructed explicitly; see~\eqref{DHformcstr}.
Moreover, as long as $X^*$ is invertible, the matrices $(J,R,Q,S,N,P,Z)$ can be constructed using~\eqref{DHformcstr},
and projected onto the feasible set $\overline{\mathbb S_{PH}}$ to obtain an initial system in PH-form.
We will refer to this initialization as `LMIs + formula'.
If one wants to obtain a better initial point, given $Q = X^*$, it possible to compute the matrices $(J,R,S,N,P)$ by solving a semidefinite program (SDP):
\begin{equation}
\min_{J,R,S,N,P} \mathcal{G}(J,R,Q,F,P,S)  \quad \text{ such that } \quad  J^T=-J \text{ and } \mat{cc}R &P\\P^T&S\rix\succeq 0,
\label{initLMIs2}
\end{equation}
while taking $Z = \mathcal{P}_{\succeq }(E^TQ)$ (as $Q=X^*$ can be ill-conditioned).
We will refer to this initialization as `LMIs + solve'.  By construction, it provides an initial point with smaller objective function value than `LMIs + formula' (at the cost of solving another SDP).

We will see that the LMI-based initializations work well when the initial system is close to being passive (that is, when $\delta^*$ is small); otherwise, it may provide rather bad initial points;  see Section~\ref{numexp} for some examples.
However, in most applications, the systems of interest are usually close to being passive (cf.\@ Introduction) hence we believe these initializations will be particularly useful in practice.

%
%
%
%

\begin{remark}[Finding the nearest stable matrix (pair)]{\rm
Our proposed algorithm to find the nearest PH system is a generalization of the algorithm of~\cite{GilS16} (resp.~\cite{GilMS17}) to find the nearest stable matrix (resp.\@ matrix pair).
In fact, it can be used on the system $(I,A,[\,],[\,],[\,])$ (resp.\@ $(E,A,[\,],[\,],[\,])$), where $[\,]$ is the empty matrix,
to recover a stable approximation of $A$ (resp.\@ $(E,A)$), not allowing (resp.\@ allowing) perturbation in $E$.
However, for the nearest stable matrix problem, the algorithm of this paper does not perform as well because authors in~\cite{GilS16} used an additional heuristic; namely, a scaling of the iterates $(J,R,Q)$ to reduce the Lipschitz constant of the objective function. Improving the performance of our algorithm using heuristics is a topic for further research.

Note that the LMI-based initializations were not introduced in~\cite{GilS16,GilMS17} (only the standard initialization was)
and could be particularly useful to obtain better solutions, especially when the input matrix (pair) is close to being stable.
}
\end{remark}

\section{Numerical experiments} \label{numexp}

Our code is available from \url{https://sites.google.com/site/nicolasgillis/} and the numerical examples presented below can be directly run from this online code. All tests are preformed using Matlab
R2015a on a laptop Intel CORE i7-7500U CPU @2.9GHz 24GB RAM. 
The algorithm runs in $O(n^3)$ operations per iteration (assuming $m \leq n$),
including projections on the set of positive semidefinite matrices, inversion of the matrix $Q$ (for general systems)
and all necessary matrix-matrix products. Hence FGM can be applied on a standard laptop with $n$ and $m$ up to a thousand.
To solve the SDPs~\eqref{initLMIs} and~\eqref{initLMIs2}, we used the interior point method SDPT3 (version 4.0)~\cite{toh1999sdpt3, tutuncu2003solving}, and CVX as a modeling system~\cite{cvx, gb08}. On a standard laptop, it can be solved for $n$ up to about a hundred (it took about 4 minutes to solve a problem~\eqref{initLMIs} with $n=100$ and $m=10$). For larger problems, it would be interesting to use first-order methods. This is a direction of further research. Therefore, as of now, for large systems ($n \gg 100$), FGM can only be used in combination with the standard initialization scheme.

In the first experiment (Section~\ref{kresex}),
we will compare our FGM with the standard projected gradient method (this is FGM restarted at each iteration) to show that FGM converges significantly faster.
In the second experiment (Section~\ref{schro}), we apply FGM on the small-scale problem from~\cite{WanZKPW10}.
In the third experiment (Section~\ref{msdsec}), we use larger mass-spring-damper systems.
In all cases, we compare the performance of the different initializations strategies from Section~\ref{initsec}.
In the fourth experiment (Section~\ref{randinit}), we compare the deterministic initializations schemes with random initializations, while in the last experiment (Section~\ref{randtest}) we compare the initialization schemes on randomly generated systems.

\subsection{ Standard system from~\cite{BoyBK89} } \label{kresex}

Consider the following standard LTI system $(E,A,B,C,D)$ from~\cite[Section 6]{BoyBK89} where $E=I_4$,
\begin{align}
&A = \mat{cccc}-0.08 & 0.83& 0& 0\\-0.83& -0.08& 0& 0\\0 &0 &-0.7& 9\\0& 0& -9 &-0.7\rix,
B = \mat{cccc}1& 0& 1&0 \\1& 0& -1&0\rix^T, \nonumber \\
&C = \mat{cccc}0.4&0&0.4&0\\0.6&0&1&0\rix, ~\text{and} ~
D = \mat{cc}0.3 &0\\0 &-0.15 \rix . \label{example1}
\end{align}
This system is asymptotically stable but not PR because the transfer function $G(s)$ does not satisfy the second condition in the Definition~\ref{def:PR_SPR_ESPR} of PR, e.g., for $s=1+2j$.

Applying FGM on~\eqref{def:dist_sph_2} with the standard initialization, 
we obtain (up to 4 digits of accuracy)
{\small{\[
\hat A=\mat{cccc}
   -0.0810  &   0.8300 & -0.4 &  -0.0104\\
   -0.8301  & -0.0799  &  0.0012 &  -0.2\\
   -0.0021  &  0.0013  & -0.8521 &   9.1\\
   -0.0146  & -0.9  & -8.9861 &  -0.8512
\rix,~
\hat B=\mat{cc}
    0.9994  &  1.8\\
    0.1  & -0.9\\
    0.9851  & -0.8691\\
   -0.0070  &  0.2
    \rix,~
    \]
 \[
\hat C=\mat{cccc}
    0.4010 &  -0.1  &  0.4185 &   0.0073\\
    0.5990  &  0.0017 &   0.8281  &  0.0158
   \rix,~
\hat D=\mat{cc}
    0.3089 &  -0.0647\\
   -0.0647 &   0.4318
         \rix.
\]}}
Figure~\ref{fig:ex1} (right) displays the evolution of the objective function:
FGM converges in about half a second while the gradient method requires about 5 seconds. However, both methods converge to the same solution.
This gives a nearby standard PR system with error
\[
{\|A-\hat A\|}_F^2 +{\|B-\hat B\|}_F^2+ {\|C-\hat C\|}_F^2+{\|D-\hat D\|}_F^2 = 0.4411.
\]
In terms of relative error for each matrix, we have
{\small{\begin{equation*}
\frac{{\|A-\hat A\|}_F}{{\|A\|}_F} = 1.68\%,~
\frac{{\|B-\hat B\|}_F}{{\|B\|}_F} = 6.60\%,~
\frac{{\|C-\hat C\|}_F}{{\|C\|}_F} = 13.41\%,~
\frac{{\|D-\hat D\|}_F}{{\|D\|}_F} = 175.62\%.
\end{equation*}}}
Note that the approximation error for $D$ is rather large.
The reason is twofold:
(i) $D$ is indefinite and the symmetric part of $D$ has to be approximated by a PSD matrix (namely, $S$) hence the relative error is at least $\frac{0.15}{\sqrt{0.3^2+0.15^2}} = 44.7\%$ (this error can be obtained by increasing the weight for $D$ in the objective function), and
(ii) the norm of $D$ compared to the other matrices (in particular $A$) is smaller hence it implicitly has less importance in the objective function.

To give more importance to $A$ and $B$, we can choose for example the weights $w_1=7/4$, $w_2= 7/4$, $w_3= 1/4$ and
$w_4=1/4$ in the objective function~\eqref{eq:weighted_obj}.
Doing so, we get another nearby standard PR system with objective function $0.13$,
and with the following relative error for each matrix
{\small{\begin{equation*}
\frac{{\|A- \tilde A\|}_F}{{\|A\|}_F} = 0.36\%,~
\frac{{\|B-  \tilde B\|}_F}{{\|B\|}_F} = 1.08\%,~
\frac{{\|C- \tilde  C\|}_F}{{\|C\|}_F} = 19.79\%,~
\frac{{\|D- \tilde D\|}_F}{{\|D\|}_F} = 197.13\%.
\end{equation*}}}

Allowing perturbations in $E$, we obtain with FGM a nearby PR system
with the approximation error $0.1812$ where weights are all equal to one (note that, as expected, it is smaller than for the more constrained standard case with error $0.4411$). Figure~\ref{fig:ex1} (left) displays the evolution of the objective function; FGM and the gradient method have a similar behavior as in the standard case.
The relative errors are
{\small{\begin{align*}
\frac{{\|E- \hat E\|}_F}{{\|E\|}_F} &= 11.49\%,~
\frac{{\|A- \hat A\|}_F}{{\|A\|}_F} = 0.19\%,~
\frac{{\|B- \hat B\|}_F}{{\|B\|}_F} = 1.81\%, ~\\
&\frac{{\|C- \hat C\|}_F}{{\|C\|}_F} = 3.32\%,~
\frac{{\|D- \hat D\|}_F}{{\|D\|}_F} = 105.24\%.
\end{align*}}}
Similarly, choosing the weights $w_1=7/4,~w_2= 7/4,~w_3= 1/4,~w_4= 1/4$ and $w_5=1$,
we obtain an objective function value of $0.07$ and the relative errors are the following
{\small{\begin{align*}
\frac{{\|E- \tilde E\|}_F}{{\|E\|}_F} &= 6.81\%,~
\frac{{\|A- \tilde A\|}_F}{{\|A\|}_F} = 0.06\%,~
\frac{{\|B- \tilde B\|}_F}{{\|B\|}_F} = 0.43\%, ~\\
&\frac{{\|C- \tilde C\|}_F}{{\|C\|}_F} = 6.05\%,~
\frac{{\|D- \tilde D\|}_F}{{\|D\|}_F} = 131.64\%.
\end{align*}}}


\begin{figure*}[h!]
\centering
\includegraphics[width=\textwidth]{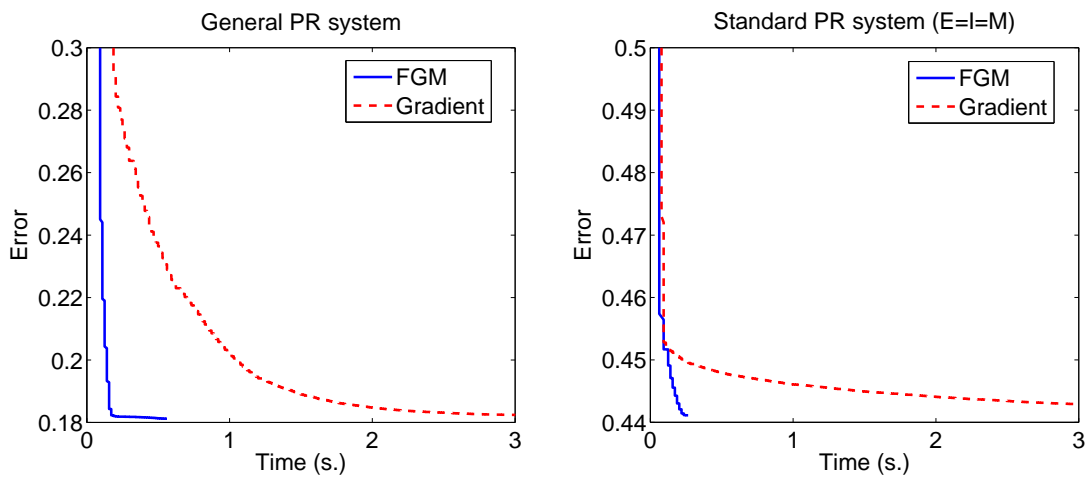}
\caption{Evolution of the objective function for FGM and GM for the system in Section~\ref{kresex}, using the standard initialization. \label{fig:ex1}}
\end{figure*}

For this example, the LMI-based initializations from Section~\ref{LMIinit} perform worse and lead to solutions with  larger error (see Table~\ref{finerrrand} that summarizes all the results).
The reason is that the original system is far from being in PH-form since $\lambda_{\min}(D) = -0.15$;
see the experiments in Sections~\ref{msdsec} and~\ref{randtest} for more insight on these initializations.


\subsection{ Descriptor system from~\cite{WanZKPW10} } \label{schro}

Consider the LTI system $(E,A,B,C,D)$ from~\cite{WanZKPW10} (see also \cite{BruS13})
where
{\small{\begin{align}
E &= \mat{cccc}
 16 &  12 &  -4 &  14 \\
 14 &   8 &   4 &  -14 \\
 -14 &   8 &  -4 &  34 \\
  6 &  -4 &   0 &  -10   \rix,
A = \mat{cccc}
 6  &  -19  &  7  &  -9  \\
 11  &  3  &  -21  &  18  \\
 25  &  -9  &  35  &  -16  \\
 -27  &  6  &  -16  &  38  \rix, \label{example2} \\
B &= \mat{cccc}  -0.6 & 1.0 & 0.2 & -0.3 \rix^T,
C = \mat{cccc} 3.2 &  1.4&    2.6&    1.4 \rix,~D = 0.105. \nonumber
\end{align}}}
The matrix pair $(E,A)$ is of index two with two finite eigenvalues $-0.5 \pm \sqrt{2} j$ hence it is not admissible.
This system is stable and remains stable if $E$ and $A$ are not perturbed.
However, it is highly sensitive to small perturbation in $E$ and $A$ because the matrix pair $(E,A)$ has Jordan block at $\infty$;
see, e.g.,~\cite{ByeN93}.
For example, replacing $E$ with $E+10^{-6}I_n$ makes the pencil $(E,A)$ unstable with an eigenvalue at 8002.
In~\cite{WanZKPW10}, $C$ is perturbed to $\hat C = [3.0876\,~ 1.4736\,~ 2.6\,~ 1.4]^T$ with ${\|C-\hat C\|}_F^2 =  0.018$ to make the
system passive.


FGM with the standard initialization obtains a nearby PR system $(\hat E,\hat A,\hat B,\hat C,\hat D)$ with approximation error 1.28 in two seconds, where the relative errors in the different matrices are
{\small{\begin{align*}
\frac{{\|E- \hat E\|}_F}{{\|E\|}_F} &=  1.26\%,
\frac{{\|A- \hat A\|}_F}{{\|A\|}_F} =  0.53\%,
\frac{{\|B- \hat B\|}_F}{{\|B\|}_F} =  24.59\%,\\
&\frac{{\|C- \hat C\|}_F}{{\|C\|}_F} =  5.04\%,
\frac{{\|D- \hat D\|}_F}{{\|D\|}_F} =  703.69\%.
\end{align*}}}
This error is not comparable to the one obtained by~\cite{WanZKPW10} because FGM provides a PH system for which $(\hat E,\hat A)$ is admissible with 4 finite eigenvalues (namely $-1.98 \pm 9.06j$, $-0.50 \pm 1.42j$).

The initialization `LMIs + solve' provides a worse but reasonable solution with error 12.47 (note that  ${\|E\|}_F^2+{\|A\|}_F^2+{\|B\|}_F^2+{\|C\|}_F^2+{\|D\|}_F^2 = 8764$), while FGM with initialization
`LMIs + formula' performs very badly with error larger than $10^6$.
The reason is that the optimal solution $X^*$ of~\eqref{initLMIs}, which is the initial value for $Q$,
is ill-conditioned (condition number of $3.4$ $10^5$).
Note that the original system is far from being in PH form since $\delta^*=2.78$ in~\eqref{initLMIs}. \\

Let us replace $E$ with  $I_4$ for which the system is not stable because we have $\max_i \text{Re} \lambda_i (E,A) = 67.6$.
This will illustrate the fact that the different initializations may perform rather differently compared to the previous example.
The LMI-based initializations provide a solution with error 2.05,
while the standard initialization provides a solution with error 263.79
(note that ${\|E\|}_F^2+{\|A\|}_F^2+{\|B\|}_F^2+{\|C\|}_F^2+{\|D\|}_F^2 = 6102$).
In this case, although the initial system is far from being stable, the LMI-based initializations perform very well
(note that $\delta^*=0.4705$ is smaller than in the previous case).


\subsection{ Mass-spring-damper system  } \label{msdsec}

Let us consider the following system: $(E,A)$ is generated as in~\cite[Section~5.3]{GilMS17}, that is,
{\small{\begin{equation}\label{eq:msd}
E=\mat{cc} V &0\\0&I_p\rix,
A=(J-R)Q,
J=\mat{cc}0&I_p\\-I_p &0\rix,
R=\mat{cc}W&0\\0&0\rix,
Q=\mat{cc}I_p&0\\0& H\rix,
\end{equation}}}
\noindent where $V\succ0$, $W\succ 0$, and $H\succ 0$ are respectively  mass, damping and
stiffness matrices of a mass-spring-damper (MSD) system.
The entries of $B \in \R^{2p,m}$  are generated using the uniform distribution in the interval $[0,1]$.
Generating each entry of $L \in \R^{m,m/2}$ using the normal distribution (mean 0, standard deviation 1), we set
$D = LL^T \succeq 0 \in \R^{m,m}$ which is rank deficient, and $C=B^TQ$.
This system clearly admits a PH-form and therefore is PR (Theorem~\ref{cor1}).
To make this system non PR, we perturb $R$ by $\tilde R=R+\Delta_R$ as in \cite{GilMS17} with
\[
\Delta_R
=
\mat{cc}0&0\\0&-\epsilon I_p\rix,
\]
for some $\epsilon > 0$.
For the numerical experiment, we take such systems of size
$n = 2p = 20$ and $m=4$, and $n=2p=40$ and $m=6$.  We use
$\epsilon = 2/(k \, n)$ for $n=20$ and $\epsilon = 1/(k \, n)$ for $n=40$, with $k=1,2,3,4$.

As shown in Table~\ref{deltastab1},
the corresponding perturbed systems do not admit a solution to the LMIs~\eqref{eq:LMI1} or, equivalently, $\delta^* > 0$ in~\eqref{initLMIs}.
However, as $\epsilon$ decreases, the system gets closer to a system admitting a PH-form in the sense that $\delta^*$ decreases.
Moreover, for smaller values of $\epsilon$, the pair $(E,A)$ is asymptotically stable;
see Table~\ref{deltastab1}.
\begin{center}
 \begin{table}[h!]
 \begin{center}
 \begin{tabular}{|c||c|c|c|c|}
 \hline
$n = 20$, $m = 4$                 &  $\epsilon = 2/n$ & $\epsilon = 1/n$ & $\epsilon = 3/(2n)$ & $\epsilon = 1/(2n)$   \\ \hline
$\delta^*$  of~\eqref{initLMIs}   &    5.4269         &   2.3647         & 1.7452            & 0.1295 \\
$\max_i \text{Re} \lambda_i(E,A)$ &    2.5749         &   0.7514         & 0.0760            & $-0.0031$ \\ \hline \hline
$n = 40$, $m = 6$   &  $\epsilon = 1/n$ & $\epsilon = 1/(2n)$ & $\epsilon = 1/(3n)$ & $\epsilon = 1/(4n)$  \\ \hline
$\delta^*$  of~\eqref{initLMIs}   &      7.7355       &  0.6577  &   0.2488  & 0.1304   \\ \hline
$\max_i \text{Re} \lambda_i(E,A)$ &      0.9004       & -0.0007  &  -0.0007  & -0.0007   \\ \hline
\end{tabular}
\caption{ Optimal value $\delta^*$ of~\eqref{initLMIs},
and largest real part of the eigenvalues of the pair $(E,A)$ for the different perturbed MSD systems.
\label{deltastab1}}
 \end{center}
 \end{table}
 \end{center}

\hspace{-.3cm} We compare four different initializations: the standard initialization (Section~\ref{staninit}), `LMIs + formula' and `LMIs + solve' (Section~\ref{LMIinit}), and the initialization using the unperturbed system, that is, taking $(J,R,Q)$ as in~\eqref{eq:msd}, $F=B$, $P=0$, $S=D$,
and $N=0$. We will refer to this last initialization as the `true' initialization as it corresponds to the groundtruth unperturbed PH system.

Figure~\ref{fig:MSDn20} (resp.\@ Figure~\ref{fig:MSDn40})
displays the evolution of the objective function values using these different initializations
for $n=20$ and $m=4$ (resp.\@ $n=40$ and $m=6$) for the different values of $\epsilon$.
The weights in the objective function are all equal to one.
\begin{figure*}[h!]
\centering
\begin{tabular}{cc}
\includegraphics[width=0.45\textwidth]{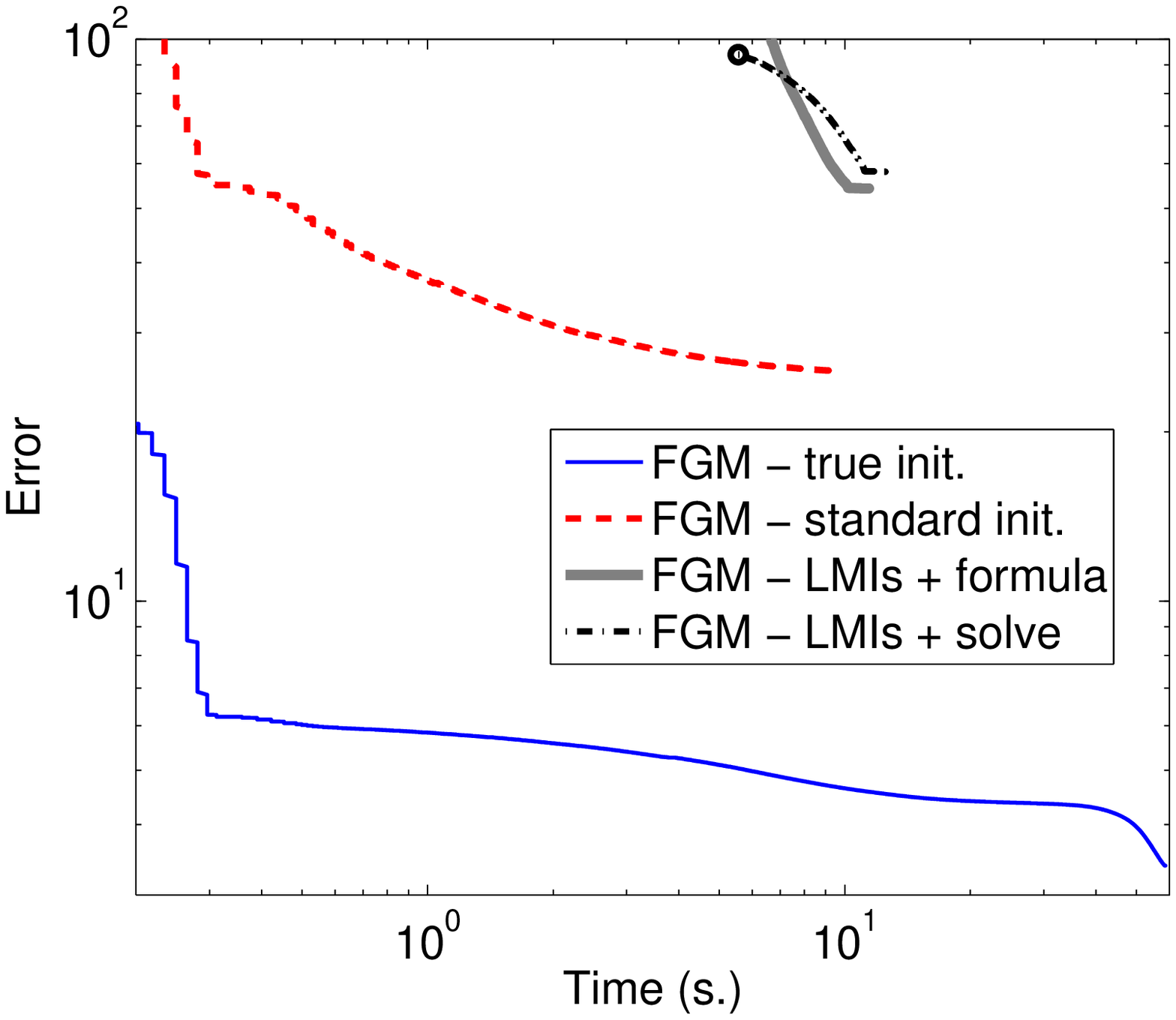} &
\includegraphics[width=0.45\textwidth]{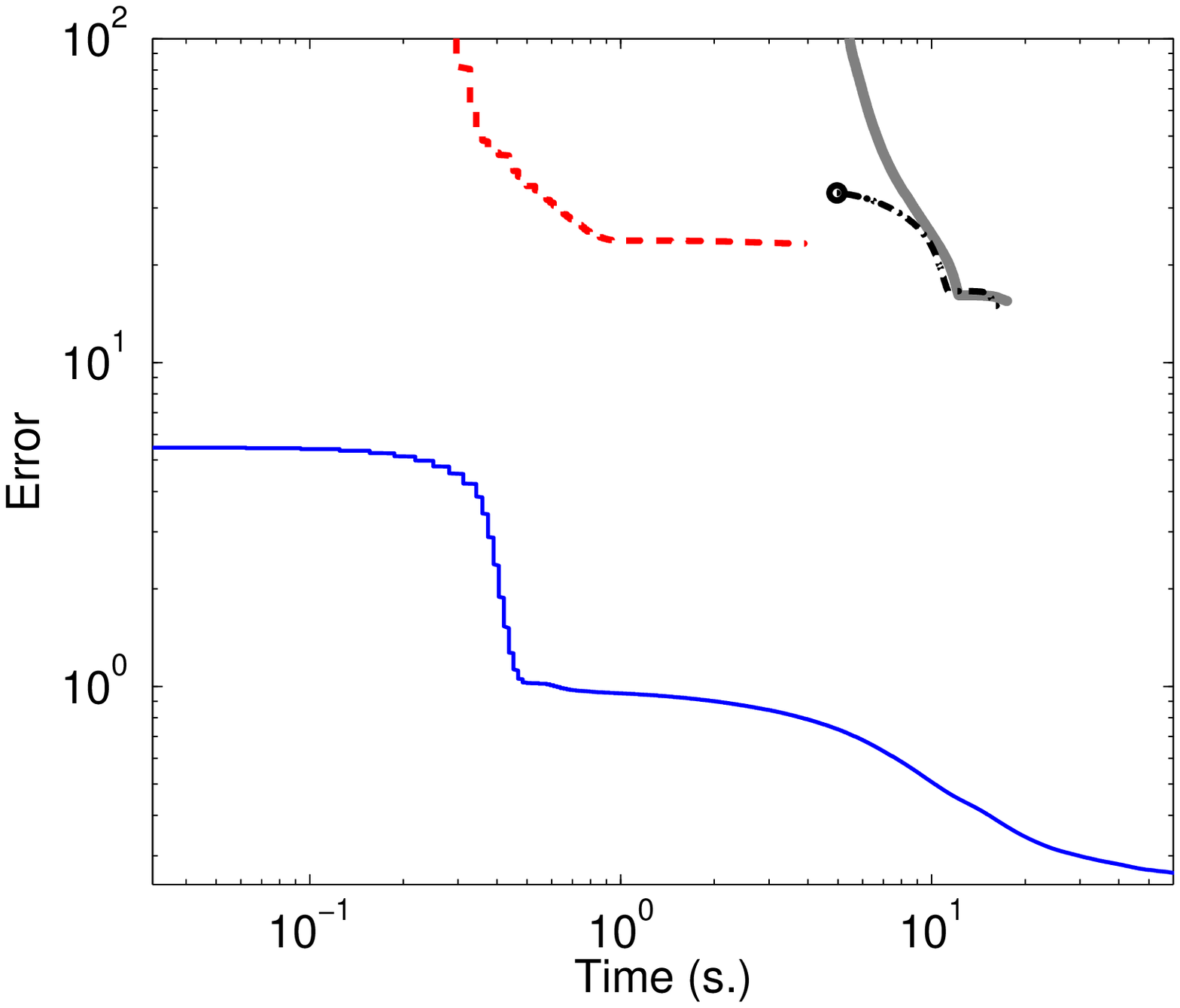} \\
\includegraphics[width=0.45\textwidth]{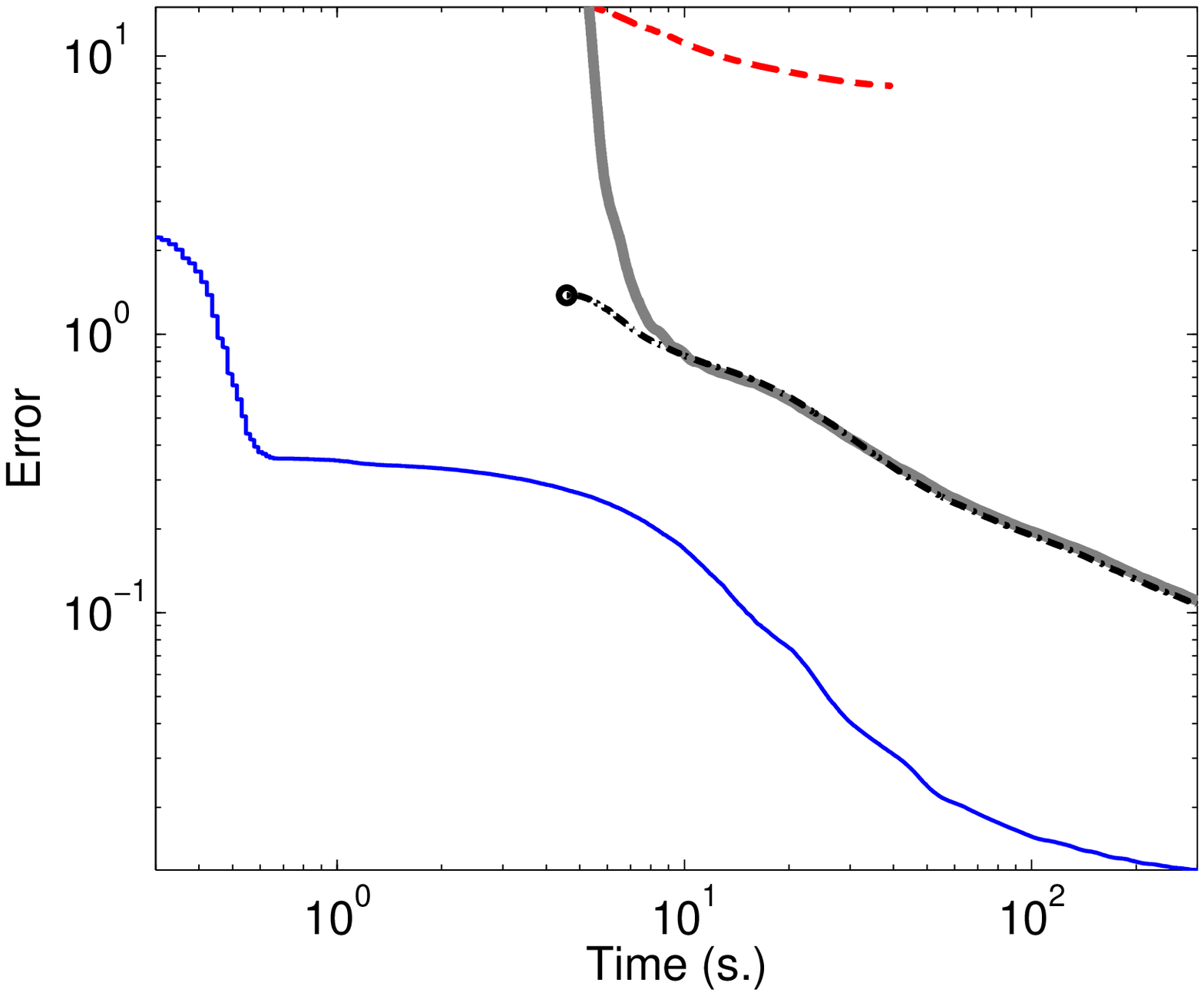} &
\includegraphics[width=0.45\textwidth]{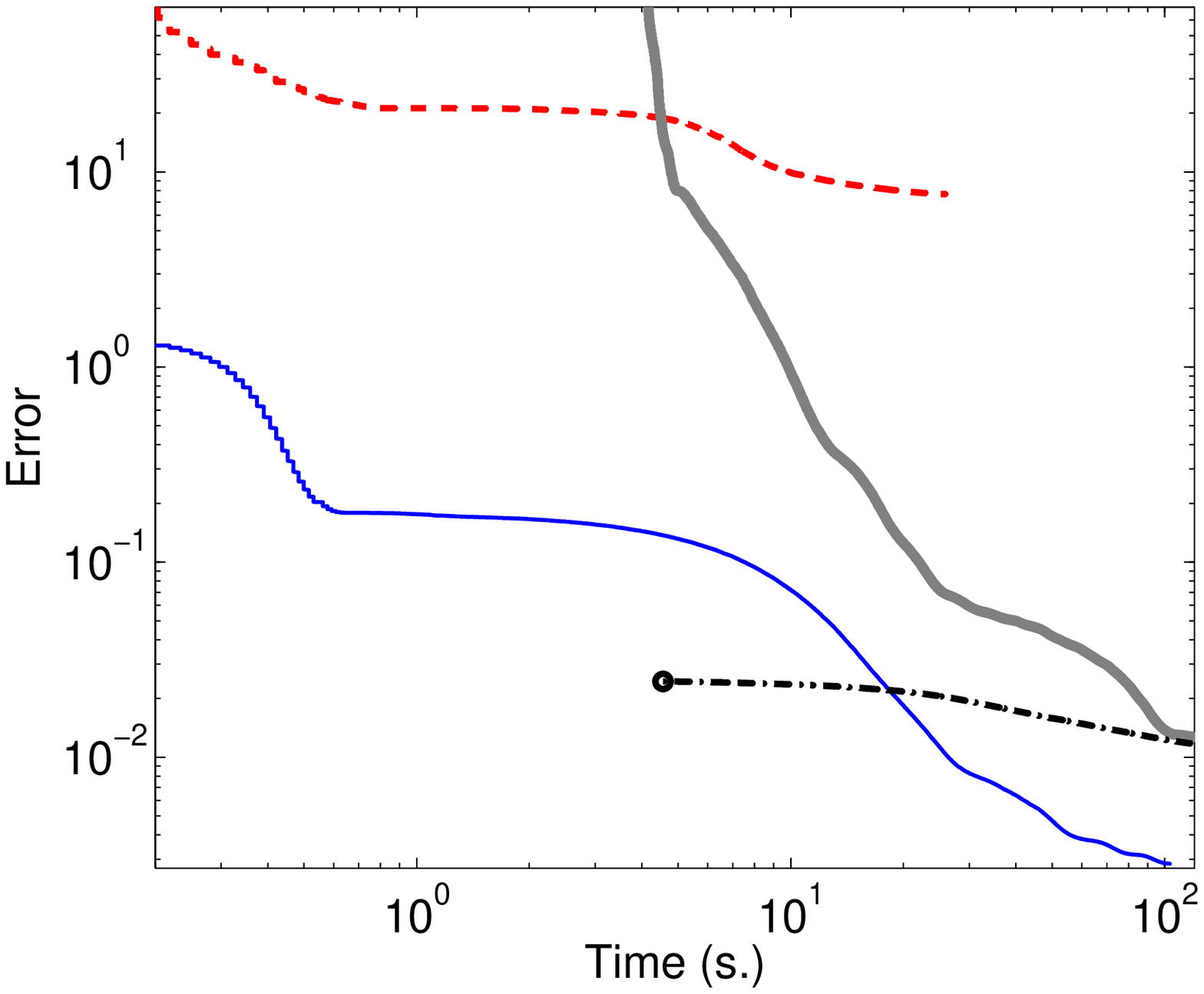} \\
\end{tabular}
\caption{Evolution of the objective function for FGM  with the different initializations
for the perturbed MSD system with $n=20$, $m=4$ and $\epsilon = 2/(nk)$, with
$k = 1$ (top left), $k = 2$ (top right), $k = 3$ (bottom left), and $k = 4$ (bottom right). \label{fig:MSDn20}}
\end{figure*}
\begin{figure*}[h!]
\centering
\begin{tabular}{cc}
\includegraphics[width=0.45\textwidth]{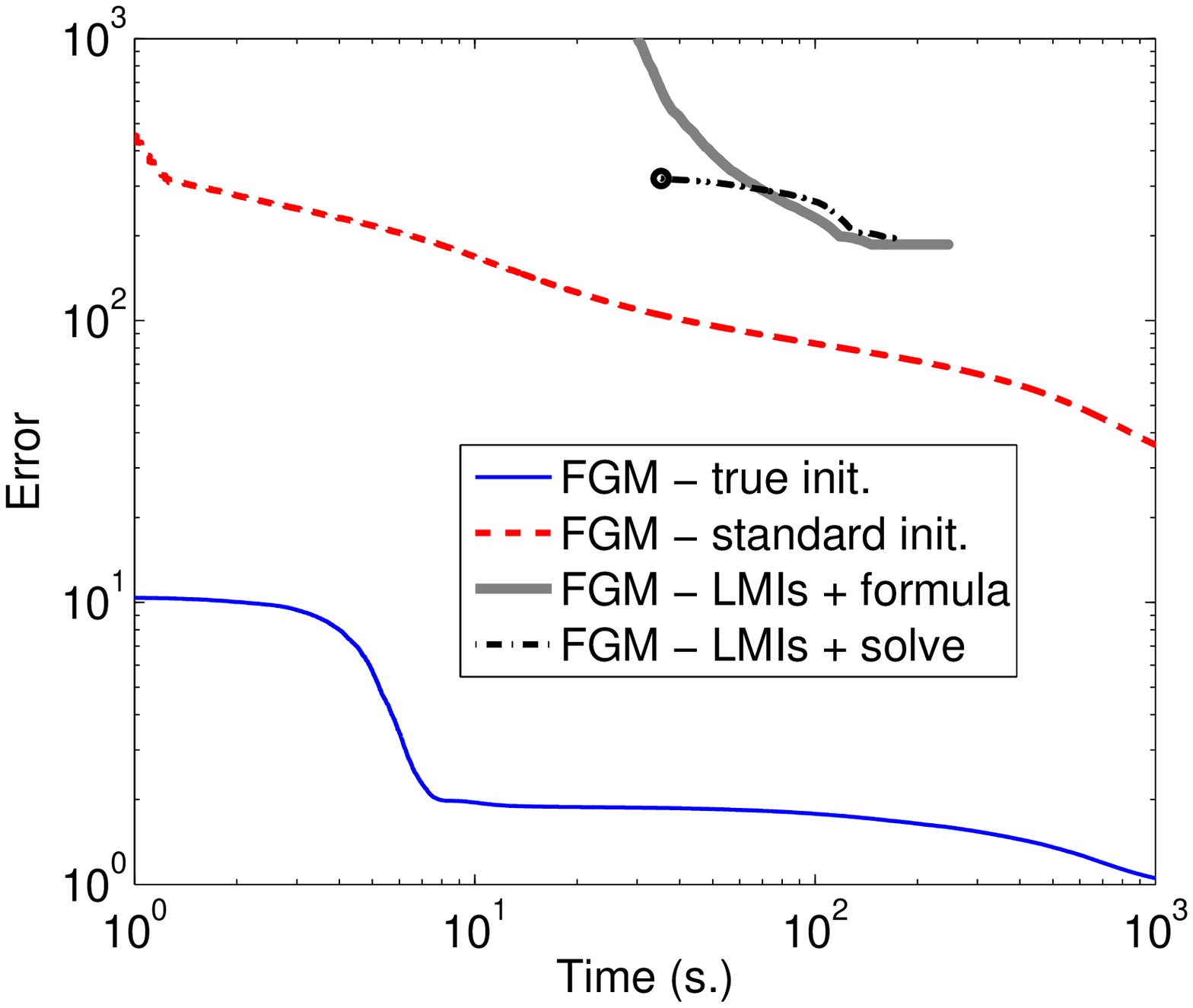} &
\includegraphics[width=0.45\textwidth]{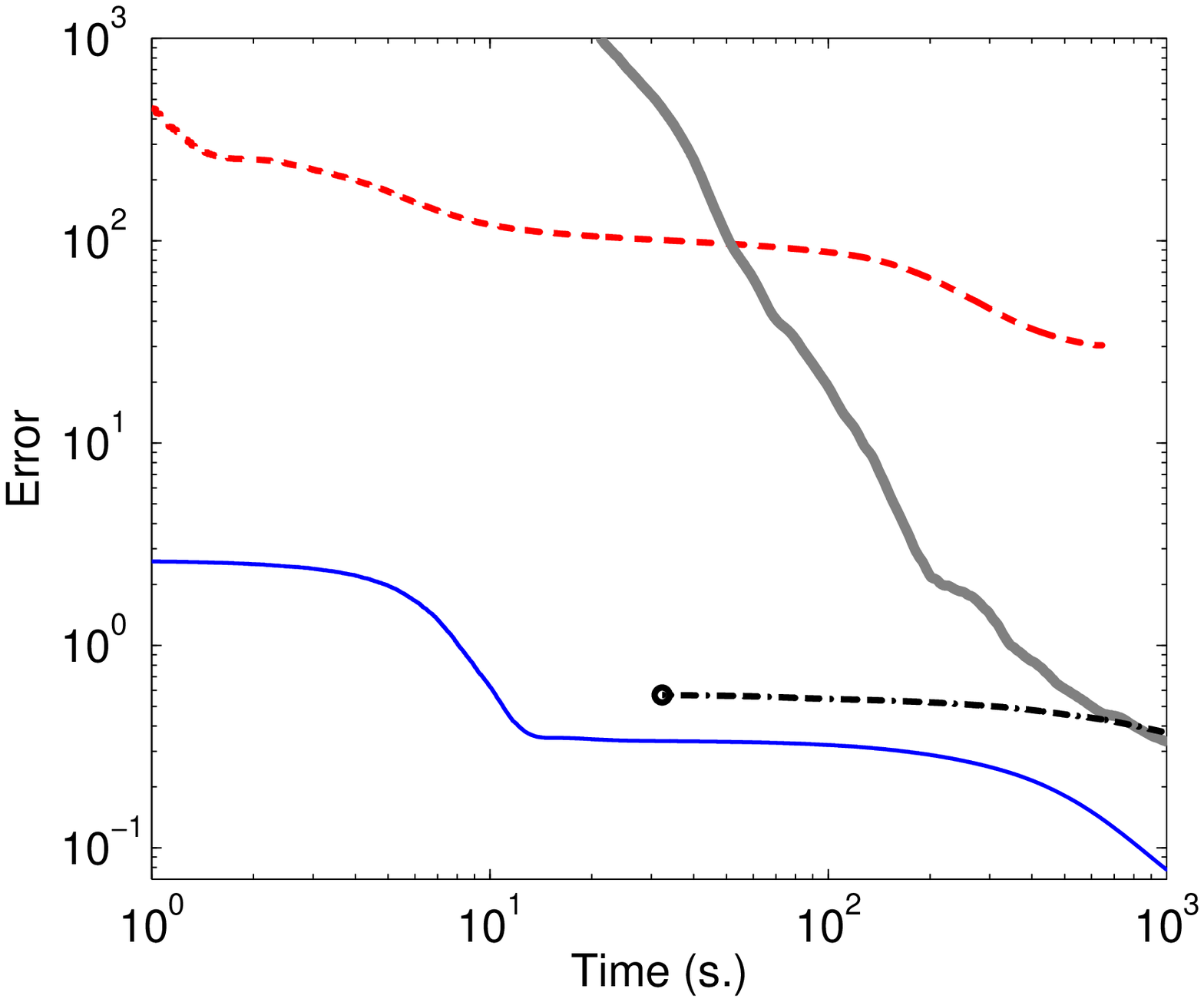} \\
\includegraphics[width=0.45\textwidth]{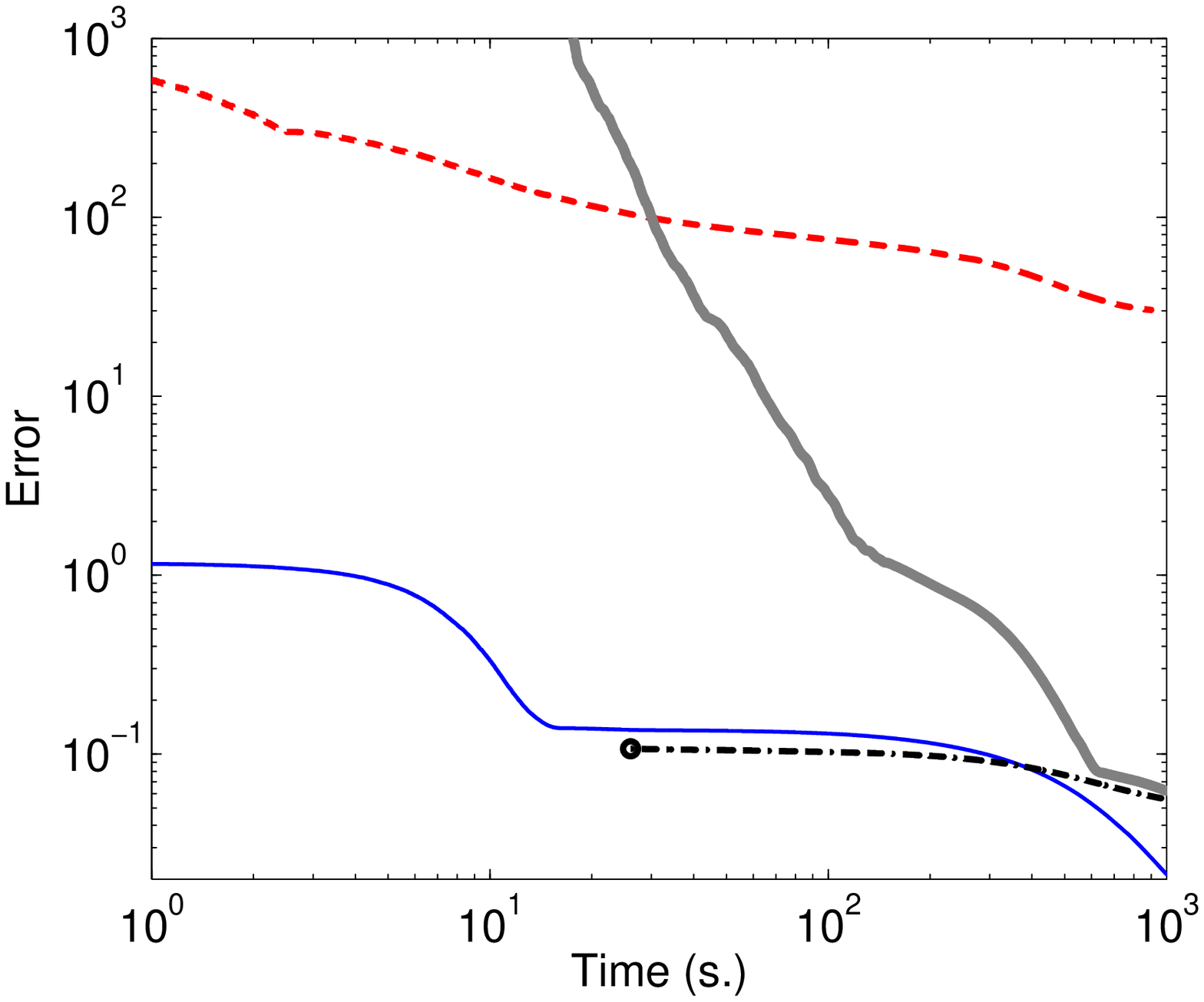} &
\includegraphics[width=0.45\textwidth]{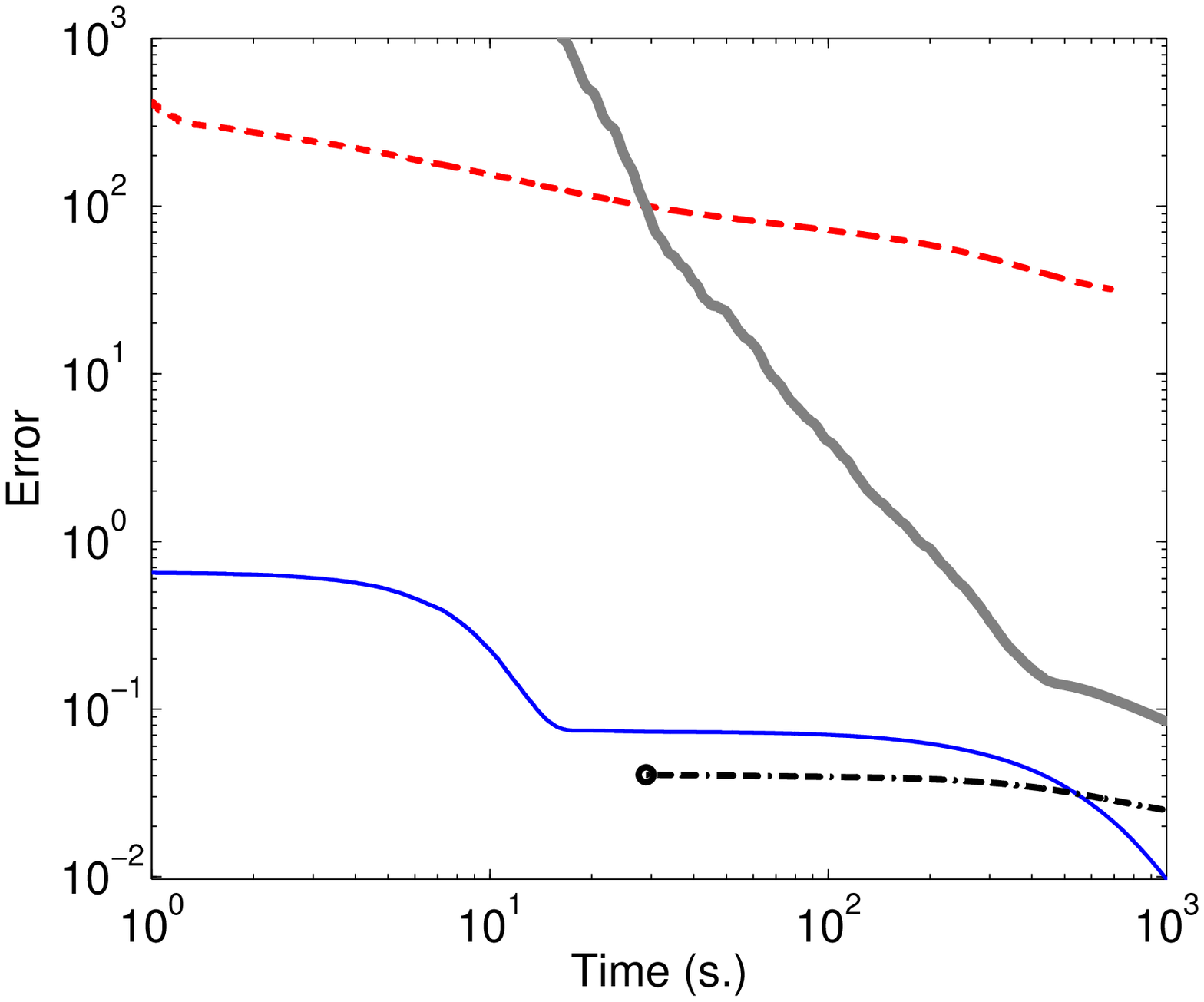} \\
\end{tabular}
\caption{Evolution of the objective function for FGM  with the different initializations
for the perturbed MSD system with $n=40$, $m=6$ and $\epsilon = 1/(nk)$, with
$k = 1$ (top left), $k = 2$ (top right), $k = 3$ (bottom left), and $k = 4$ (bottom right).\label{fig:MSDn40}}
\end{figure*}

Table~\ref{finerrrand} gives the final error obtain by the different algorithms with a maximum time limit of 300 seconds for $n=20$ and 1000 seconds for $n=40$. Before we comment these results, it is important to put these numbers into perspective:
we have ${\|E\|}_F^2+{\|A\|}_F^2+{\|B\|}_F^2+{\|C\|}_F^2+{\|D\|}_F^2 = 5456$
(resp.\@ $= 43078$)  for $n = 20$ (resp.\@ $n=40$).
Hence, for example, the largest error of $58.00$ (resp.\@ $196.11$) of `LMIs + solve' for $n=20$ (resp.\@ $n=40$) and $k=1$ corresponds to a reasonable approximation although it is much larger than for some other approaches.
%

For both dimensions, we observe a similar behavior of FGM for the different initializations:
\begin{itemize}

\item FGM converges in most cases at a sublinear rate; see Figures~\ref{fig:MSDn20} and~\ref{fig:MSDn40} where the objective function values decrease roughly linearly in a logarithmic time scale.

\item For the true initialization, FGM recovers a solution with the smallest error after sufficiently many iterations.
This is rather natural since the initialization corresponds to the original unperturbed PH system.

\item For the standard initialization, FGM converges to systems with error larger than with the true initialization, and gets stuck is some local minima.
This illustrates the importance of choosing good initial points although, as mentioned above, in terms of relative error, these solutions still provide good approximations. For high perturbations ($k=1$), it provides significantly better solutions than the LMI-based initializations.

\item For `LMIs + formula', the initial error is rather high (even for a small perturbation $\epsilon$),
because the formula~\eqref{DHformcstr} does not provide a good estimate of the PH-form when the system is not PH.

For large $\epsilon$ ($k=1$), it is not able to recover a solution close to the one obtained with the true initialization.

For $\epsilon$ sufficiently small, and after sufficiently many iterations, it is able to recover a solution with error similar
to that of FGM initialized with `LMIs + solve' and close to that obtained with the true initialization.
`LMIs + formula' and `LMIs + solve' often converge to similar solutions which can be explained
by the fact that both initializations use the same initial $Q$ and $Z$.

\item For `LMIs + solve', FGM is able to recover better and better solutions as $\epsilon$ decreases.
For the largest $\epsilon$ ($k=1$), it performs worse than the standard initialization.
For the smallest $\epsilon$ ($k=4$), the initial point obtained with `LMIs + solve'  has smaller error than the true initialization.
Since the initialization `LMIs + solve' computes the optimal values for $J,R,P,S,N$ and $F$ for fixed $Q$ (at a larger initial computational cost),
it is not surprising that it has a much lower initial error than `LMIs + formula'.

\end{itemize}

\begin{remark}\label{rem:strictPR}[Nearest strict PH system]
{\rm
The PH system obtained with FGM are not necessarily strict since the cost matrix $K$ can be rank deficient.
For example, with the true initialization for $n=20$ and $k=1$, it has 11 eigenvalue with  modulus smaller than $10^{-12}$. It is possible to impose the system to be strict (hence admissible, and ESPR if $D+D^T \succ 0$; see Theorem~\ref{thm:main_result})
using a lower bound on the eigenvalues of $K$ which does not make the projection step more complicated.
Note that it is also possible to use a lower bound $\nu$ for the eigenvalues of $Z$ to have $Q$ invertible (as long as it is initialized with an invertible matrix).
In fact, the objective function is guaranteed to decrease under the updates of FGM, hence the term ${\|E^T - ZQ^{-1}\|}_F^2$ remains bounded
which guarantees $Q$ to be invertible since we would have $Z \succeq \nu I_n$.
We have included this option in our code.
}
\end{remark}

\subsection{Random initializations} \label{randinit}

So far, we have only used deterministic initializations.
As expected, in some cases, they do not lead to good solutions (see for example Table~\ref{finerrrand}).
Therefore, an important direction for further research is to design new initialization schemes, possibly depending on the problem at hand. For example, we have seen in the previous sections that if the perturbed system is close to being passive,
then the LMI-based initializations perform well. We will confirm this behavior on randomly generated systems in the next section.
A simple initialization scheme is to use random matrices.
In this section, we perform some numerical experiments to get some insight on whether this allows to recover good solutions for the systems presented in the previous sections.

Defining a Gaussian matrix as a matrix whose entries are generated randomly using the normal distribution $N(0,1)$ (we used the function \texttt{randn} in Matlab),
we initialize the variables as follows:
$Q$ is the product of a Gaussian matrix with its transpose so that it is full rank\footnote{For $n$ large, the conditioning of $Q$ can be bad ($\gg 10^4$). In that case, we compute the SVD of $Q$, and set the smallest singular values of $Q$ to $\sigma_{\max}(Q)/\kappa$ so that the conditioning of $Q$ is a most $\kappa$ (we used $\kappa=10^4$).},
$J$, $Z$ and $\mat{cc} R & S \\  S^T & P \rix$
are Gaussian matrices projected onto the feasible set,
$F$ is chosen optimal using~\eqref{formulF}.

Table~\ref{finerrrand} summarizes the results using 100 random initializations with a time limit of 10 seconds for Algorithm~\ref{fastgrad}.
However, the comparison is not fair for the MSD systems of Section~\ref{msdsec} that were run 300 and 1000 seconds respectively for $n=20$ and $n=40$, but this would take a long time to run 100 times.
Hence, for these systems, Algorithm~\ref{fastgrad} was run for 300 and 1000 additional seconds only on the best solution obtained after 10 seconds among the 100 random initial points.
\begin{center}
 \begin{table}[h!]
 \begin{center}
 \begin{tabular}{|c||c|c|c|c|c|}
 \hline             &  random    &  standard & LMI+form. & LMI+sol.  & true  \\ \hline
 \hline
\eqref{example1}
&  \textbf{0.4411}   & \textbf{0.4411} &  4.07  &  3.26  &  /  \\
$\hat{E}=I$  &  (0.52 $\pm$ 0.35)  & &   &   &    \\ \hline
\eqref{example1}
&  \textbf{0.1812}   & \textbf{0.1812} &  2.27  & 2.22    &  /  \\
&  (3.38 $\pm$ 11.9)   & &   &   &    \\ \hline
\hline
\eqref{example2}
& \textbf{0.62}       &  1.28  & $>10^6$  &   12.47  &   / \\
&  (172 $\pm$ 164)   & &   &   &    \\ \hline
\eqref{example2}
& \textbf{1.39} &  263.79  & 2.05  & 2.05    &  /  \\
$E=I_4$   &  (63.43 $\pm$ 148)    & &   &   &    \\ \hline
 \hline
$n$=20, $k$=1
&  \textbf{1.90} & 26.69 & 54.22  &  58.00   &  3.38  \\
&  (37.57 $\pm$ 22.20)    & &   &   &    \\ \hline
$n$=20, $k$=2
&  1.48 & 23.31 &  15.50  &  14.97   &  \textbf{0.25}  \\
&  (34.52 $\pm$ 25.56)    & &   &   &    \\ \hline
$n$=20, $k$=3
&  \textbf{0.01} & 7.82 & 0.11  &   0.11  &   \textbf{0.01} \\
&  (32.54 $\pm$ 30.38)    & &   &   &    \\ \hline
$n$=20, $k$=4
&  0.92 & 7.68  &  9.21 $10^{-3}$ &  8.96 $10^{-3}$   &   \textbf{2.85 $10^{-3}$}  \\
&  (33.91 $\pm$ 24.48)    & &   &   &    \\ \hline
\hline
$n$=40, $k$=1
&  66.19 & 36.26 & 185.93  &  196.11   &  \textbf{1.05}  \\
&  (1708 $\pm$ 613)    & &   &   &    \\ \hline
$n$=40, $k$=2
& 13.38  & 30.42 & 0.33  &   0.37  &   \textbf{0.08} \\
&  (1815 $\pm$ 673)    & &   &   &    \\ \hline
$n$=40, $k$=3
& 299.65  & 30.31 & 0.06  &  0.06   &   \textbf{0.02}\\
&  (1578 $\pm$ 653)    & &   &   &    \\ \hline
$n$=40, $k$=4
&  18.32 & 32.05 &  0.09  &  0.03   &  \textbf{0.02}   \\
&  (1704 $\pm$ 603)    & &   &   &    \\ \hline
\hline
\end{tabular}
\caption{ Comparison of the error obtained by the different initialization schemes on the systems from Section~\ref{kresex}, \ref{schro} and~\ref{msdsec}.
For the 100 random initializations, we also report the mean and standard deviation in brackets (for the error obtained within 10 seconds). Bold indicates the lowest error among all initializations.
\label{finerrrand}}
 \end{center}
 \end{table}
 \end{center}

In many cases, random initialization identifies good solutions; in fact, it achieves the lowest error compared to  the three deterministic approaches for all problems of size $n \leq 20$,
except for the MSD system with $n=20$ and $k=4$.
It even competes similarly as the `true' initialisation for the MSD systems
(providing a lower error for
$n=2$ and $k = 1$, and the same error for  $n=2$ and $k = 3$).
However, for the larger MSD systems with $n = 40$,
it is not able to compete with the LMI-based initializations that perform well in this situation (for $k \geq 2$).
The reason is that the number of local minima increases and the algorithm converges in the basin of attraction of worse local minima: not surprisingly, the standard deviation of the errors obtained with random initializations increases with the dimension of the problem.
For example, for $n=40$ and $k = 3$, the error obtained is rather high, namely 299.65. However, running the algorithm with 100 other random initializations, we obtained an error of 28.07.
For larger problem, a direction for further research is therefore to design more sophisticated heuristics to identify better solutions and avoid the basin of attraction of bad local minimizers.

\subsection{ Randomly generated systems  } \label{randtest}

In the previous sections,
we observed that
\textit{if the perturbed system is close to being passive, then the LMI-based initializations are able to recover a passive system that is closer to the perturbed one than the initial true passive system.}
This means that the LMI-based initializations provide an initial point that is in the basin of attraction of a very good local minimum (possibly the global minimum, although this is difficult to verify).
Intuitively, the reason is that the LMIs~\eqref{initLMIs2} are only slightly perturbed hence the solution will be close to the solution of the system for the unperturbed passive system.
This is closely related to perturbation analysis of optimization problems~\cite{bonnans2013perturbation}.
 Although we are not able to prove this important observation (which would be a very interesting direction of further research), we perform in this section additional numerical experiments to support it.


To generate systems randomly, we use the same strategy as in the previous section to obtain a PH system, and set $N=0$.
Then, given the parameter $\delta$, we perturb $R$ and $S$ in the same way as follows: given a matrix $X$,
\begin{itemize}
\item Compute its singular values decompositions $[U,\Sigma,V]$ with $X = U\Sigma V^T$,

\item Set the smallest singular value in $\Sigma$ to zero (note that, in many cases, $R$ and $S$ already have a singular value equal to zero since they are the projection of Gaussian matrices onto the PSD cone).

\item Compute $\Sigma' = \Sigma - \delta \sigma_{\max}$ where $\sigma_{\max}$ is the largest entry of $\Sigma$, and replace $X$ with by $\tilde{X} = U \Sigma' V$.
\end{itemize}
With this procedure, the perturbed $R$ and $S$ do not belong to the PSD cone, for any $\delta > 0$. Since $S$ is the symmetric part of $D$, this will generate a perturbed system that does not admit a PH form.
The parameter $\delta$ is chosen such that a certain relative distance $\epsilon$ is achieved between the randomly generated PH system $(E,A,B,C,D)$ and
its perturbation $(\tilde{E},\tilde{A},\tilde{B},\tilde{C},\tilde{D})$:
\[
 \frac{{\|(A-\tilde{A},E-\tilde{E},B-\tilde{B},C-\tilde{C},D-\tilde{D})\|}_F}{ {\|(\tilde{A},\tilde{E},\tilde{B},\tilde{C},\tilde{D})\|}_F}  = \epsilon.
\]
(We used a simple bisection scheme to find $\delta$, given $\epsilon$.)
We will compare the three deterministic initialization schemes as in the previous sections (standard, LMI+formula, LMI+solve) along with the `true' initialization which is  the original unperturbed randomly generated PH system. Hence, the `true' initialization is guarantees to achieve a relative error smaller $\epsilon$, since Algorithm~\ref{fastgrad} is guaranteed to decrease the objective function at each step.

We generate systems with $n=20$ and $m=5$ as described above, and Table~\ref{finrandtest} summarizes the average relative error in percent among 10 such randomly generated systems, with a time limit of 100 seconds, for different values of $\epsilon$.

The standard initialization scheme consistently performs worse than the other approaches.
LMIs+formula performs rather well, similarly as the `true' initialization,
while LMIs+solve surprisingly performs best in all scenarios, being able to identify better solutions than the `true' initialization.

This experiment allows us to confirm our previous observation: as the perturbed system gets closer to a PH system, LMI-based initialization are able to recover better and better solutions.
\begin{center}
 \begin{table}[h!]
 \begin{center}
 \begin{tabular}{|c||c|c|c|c|}
 \hline
 $\epsilon$   &  standard init.   &  LMIs + formula  &  LMIs + solve &    true init.   \\
 \hline
   50\%   &  17.95 &  11.21 &   2.26   &  13.38 \\
   10\%   &  12.37 &   0.89 &   0.24   &   0.87 \\
    1\%   &  12.60 &   0.52 &   0.0071 &   0.21 \\
    0.1\% &  15.89 &   0.19 &   0.0028 &   0.035 \\
\hline
\end{tabular}
\caption{ Average relative error in percent for the different initialization schemes for randomly generated systems of size $n=20$ and $m=5$.
\label{finrandtest}}
 \end{center}
 \end{table}
 \end{center}

\section{Conclusion and further research}

In this paper, we have proposed the first algorithm to tackle the nearest positive-real system problem that allows the perturbation of all matrices $(E,A,B,C,D)$ describing an LTI system.
Our approach combines a reformulation of positive-real systems as port-Hamiltonian (PH) systems and a fast gradient method (FGM).
We have illustrated the effectiveness of our approach on several examples.
In particular, we observed that {if the initial system is close to being PH, then the proposed LMI-based initializations allow to recover nearby PH systems}.
An interesting direction for further research would be to characterize this rigorously, e.g., providing error bounds for the LMIs-based initializations.
Another observation is that FGM is sensitive to initialization and does not always converge fast (often in sublinear rate),
hence further research includes for example the design of
(i)~new initialization schemes,
(ii)~more efficient algorithms (e.g., using second-order information), and
(iii)~globalization approaches to escape local minima.

\section*{Acknowledgement}

The authors thank the anonymous reviewers for their insightful comments which helped improve the paper.

\small

\bibliographystyle{siam}
\bibliography{GilS}

\end{document}